\documentclass[reqno,12pt]{amsart}
\usepackage{amssymb}
\usepackage[usenames, dvipsnames]{color}
\usepackage{hyperref}

\usepackage{verbatim}
\usepackage{todonotes}

\newtheorem{theorem}{Theorem}[section]
\newtheorem{corollary}[theorem]{Corollary}
\newtheorem{lemma}[theorem]{Lemma}
\newtheorem{proposition}[theorem]{Proposition}

\theoremstyle{definition}
\newtheorem{remark}[theorem]{Remark}

\theoremstyle{definition}
\newtheorem{definition}[theorem]{Definition}

\theoremstyle{definition}
\newtheorem{assumption}[theorem]{Assumption}

\makeatletter
\def\dashint{\operatorname%
{\,\,\text{\bf--}\kern-.98em\DOTSI\intop\ilimits@\!\!}}
\makeatother

\def\bR{\mathbb{R}}

\def\bZ{\mathbb{Z}}

\def\bH{\mathbb{H}}

\def\bL{\mathbb{L}}

\def\fH{\mathfrak{H}}

\def\cD{\mathcal{D}}

\def\cH{\mathcal{H}}

\def\cM{\mathcal{M}}

\def\cS{\mathcal{S}}
\def\cT{\mathcal{T}}

\begin{document}

\title[Trace theorem and BVP in weighted Sobolev spaces]{Trace theorem and non-zero boundary value problem for parabolic equations in weighted Sobolev spaces}

\author[D. Kim]{Doyoon Kim}
\address[D. Kim]{Department of Mathematics, Korea University, 145 Anam-ro, Seongbuk-gu, Seoul, 02841, Republic of Korea}

\email{doyoon\_kim@korea.ac.kr}

\thanks{D. Kim and K. Woo were supported by Basic Science Research Program through the National Research Foundation of Korea (NRF) funded by the Ministry of Education (2019R1A2C1084683).}

\author[K.-H. Kim]{Kyeong-Hun Kim}
\address[K.-H. Kim]{Department of Mathematics, Korea University, 145 Anam-ro, Seongbuk-gu, Seoul, 02841, Republic of Korea}

\email{kyeonghun@korea.ac.kr}

\thanks{K.-H. Kim was supported  by the National Research Foundation of Korea(NRF) funded by the Korea government(MSIT) (NRF-2020R1A2C1A01003354)}

\author[K. Woo]{Kwan Woo}
\address[K. Woo]{Department of Mathematics, Korea University, 145 Anam-ro, Seongbuk-gu, Seoul, 02841, Republic of Korea}

\email{wkwan@korea.ac.kr}



\begin{abstract}
We present weighted Sobolev spaces $\widetilde{\fH}_{p, \theta}^{\gamma}(S, T)$ and prove a trace theorem for the spaces.
As an application, we discuss non-zero boundary value problems for parabolic equations.
The weighted parabolic Sobolev spaces we consider are designed, in particular, for the regularity theory of stochastic partial differential equations on bounded domains.
\end{abstract}

\maketitle

\section{Introduction}
\label{intro}
We introduce appropriate Sobolev type spaces along with trace and extension results for the weighted Sobolev spaces $H_{p, \theta}^{\gamma}\left( \Omega \right)$.
One needs such results when dealing with non-zero (lateral) boundary value problems for parabolic equations in the framework of weighted Sobolev spaces. Krylov first introduced the weighted Sobolev spaces $H_{p, \theta}^{\gamma}\left( \Omega \right)$ in \cite{MR1708104} for $\Omega = \bR^n_+ = \{ x = (x_1, x'): x_1 > 0, x' \in \bR^{n-1} \}$ and $\gamma \in \bR$ in order to establish an $L_p$-theory for stochastic partial differential equations (SPDEs). See, for instance, \cite{MR1720129}.
(Later in the introduction comes an explanation about why one needs $H^{\gamma}_{p,\theta}(\Omega)$-type spaces in the theory of PDEs and SPDEs.)
Note that if $\gamma$ is a non-negative integer, $H_{p, \theta}^{\gamma}$ can be characterized as
\[
H_{p, \theta}^{\gamma}(\Omega) =  \{ u: \rho(x)^{|\alpha|}D^{\alpha}u \in L_{p, \theta}(\Omega), \quad 0 \le |\alpha| \le \gamma \}, 
\]
where $\rho(x)=\text{dist}(x,\partial \Omega)$ and $L_{p, \theta}(\Omega)$ is the $L_p$ space with the measure $\rho^{\theta - n} \, dx$, and $\mathfrak{H}_{p, \theta}^\gamma(S, T)$ is the parabolic counterpart of $H_{p, \theta}^{\gamma}(\Omega)$.
See Section \ref{setting} for precise definitions of these spaces.
In particular, as is shown in \cite{MR1708104, MR1796011}, the class of infinitely differentiable functions with compact support in $\Omega$ is a dense subset of $\mathfrak{H}_{p, \theta}^\gamma(S, T)$.
This means that every function in $\fH_{p, \theta}^\gamma(S, T)$ has a zero boundary value whenever the trace (the restriction of a function on the boundary) makes sense.
Thus, one can say that in \cite{MR1708104} and related papers \cite{MR2111792, MR2990037, MR2354494, MR3147235, MR3318165, arXiv:1809.01325}, the authors deal with parabolic PDEs with the zero boundary condition.
To the best of the authors' knowledge, parabolic equations with non-zero boundary conditions in the above $\fH_{p, \theta}^\gamma$-type spaces have not been considered.
To cover this, we need to identify not only appropriate classes of functions for boundary values, but also new solution spaces in which functions have nontrivial boundary values, so that the traces of functions in the new solution spaces belong to the aforementioned function spaces on the boundary.
For elliptic problems, the corresponding weighted Sobolev spaces and their properties are presented in \cite{MR1708104, MR1760536, MR1796011, MR2111792, MR2386392}.
In particular, in \cite{MR2386392} one can find a trace theorem for weighted Sobolev spaces and a solvability result for elliptic equations with non-zero boundary conditions.

In this paper, as solution spaces for non-zero parabolic boundary value problems, we present weighted Sobolev spaces $\widetilde{\fH}_{p, \theta}^{\gamma}(S, T)$ (see Definition \ref{sol_space}).
Then, we prove that the parabolic Slobodeckij spaces $W_{p}^{\mathsf{s}/2, \mathsf{s}}((S, T) \times \partial\Omega)$ (see Definition \ref{trace_space}) are the function spaces for the boundary values.
More precisely, we prove a trace theorem (Theorem \ref{thm0716_01}) for $\widetilde{\fH}_{p, \theta}^{\gamma}(S, T)$, $\gamma \geq 1$, when $\theta \in (n-1, n-1+p)$ and $\Omega = \bR^n_+$ or $\Omega$ is a bounded Lipschitz domain, i.e., $\partial \Omega \in C^{0,1}$, so that the restriction of $u$ on $(S,T) \times \partial\Omega$ belongs to $W_{p}^{\mathsf{s}/2, \mathsf{s}}((S, T) \times \partial\Omega)$ for $u \in \widetilde{\fH}_{p, \theta}^{\gamma}(S, T)$ and each $g \in W_{p}^{\mathsf{s}/2, \mathsf{s}}((S, T) \times \partial\Omega)$ can be extended to a function in $\widetilde{\fH}_{p, \theta}^{\gamma}(S, T)$ along with the boundedness of the trace (restriction) and extension operators. As also noticed in \cite{MR2386392}, due to the characteristic of $\widetilde{\fH}_{p, \theta}^{\gamma}(S, T)$, the parameter $\mathsf{s}$ in $W_{p}^{\mathsf{s}/2, \mathsf{s}}$ is independent of $\gamma$.
Regarding the boundedness of the trace operator, the main difficulty is due to the fact that $u_t$ has low regularity when $u \in \widetilde{\fH}_{p, \theta}^1(S, T)$ ($\gamma = 1$). See Remark \ref{rem0426_1}.
To overcome this, we take an appropriate integral representation of $u$ based on the argument in \cite{MR0152793}. Also see \cite{MR1257932}.
In our case, we use mollifications instead of kernels in the integral form used in \cite{MR0152793} and \cite{MR1257932}.
We remark that using mollifications turns out to be very convenient when dealing with not only traces of functions, but also embeddings. See, for instance, \cite{arXiv:2103.03410}. For the boundedness of extension operator, we obtain weighted Sobolev norm estimates for solutions to the heat equation with non-homogeneous boundary values. 
As a direct application of our trace and extension results, we establish the unique solvability of second-order parabolic equations in divergence form with non-zero lateral boundary conditions in the weighted Sobolev spaces $\widetilde{\fH}_{p, \theta}^1(S, T)$ under the zero initial condition. As one can expect, the unique solvability is established by combining the trace theorem in this paper with the known result for parabolic equations with the homogeneous boundary condition, that is, as mentioned earlier, the equations with solutions in $\fH_{p, \theta}^1(S, T)$.
We also obtain the corresponding result for non-divergence form equations in $\widetilde{\fH}_{p, \theta}^2(S, T)$. In a forthcoming paper, we will extend the present results to initial traces and non-zero initial value problems.

Trace theorems for Sobolev type spaces have been studied in many papers. 
See \cite{MR0241822, MR1257932, MR1945779, MR2318575, MR2401510, MR3273638} and references therein. For recent results, we refer the reader to \cite{MR4072650}, which contains trace theorems for anisotropic mixed-norm Sobolev, Bessel potential, Triebel-Lizorkin, and Besov type spaces with space-time power weights.
Nevertheless, our results are not covered by those, for instance, in \cite{MR4072650} in that the spaces considered there are different from ours.
In \cite{MR4072650}, the author considers function spaces for equations of non-divergence type, while our results contain trace and extension results for $\widetilde{\fH}_{p, \theta}^1(S, T)$, which is an appropriate function space for divergence type parabolic equations in weighted Sobolev spaces such as $\mathfrak{H}_{p, \theta}^\gamma(S, T)$.
Another notable difference is that the proofs in \cite{MR2318575}, which is one of the main references in \cite{MR4072650}, are based on an interpolation theory of vector-valued function spaces with a semigroup approach, whereas our proofs are more elementary and self-contained. However, as mentioned above, the function spaces considered in \cite{MR4072650} are very general and the non-zero initial traces are also considered.

Among many weighted Sobolev spaces for various purposes in the theory of PDEs, $H^{\gamma}_{p,\theta}$-type spaces were necessitated by, mainly, the theory of SPDEs.
As explained in \cite{MR1262972}, the Sobolev spaces without weights turn out to be trivially inadequate for SPDEs. This is because,   unless certain compatibility condition is fulfilled, the second and higher derivatives of solutions to SPDEs blow up substantially fast near the boundary, and this blow-up is inevitable even on $C^{\infty}$ domains. On the other hand,  it turns out that such blow-up behavior can be described very accurately by a weight system based on appropriate powers of  the distance to the boundary, that is, by $H^{\gamma}_{p,\theta}$-type spaces. For SPDEs in weighted Sobolev spaces, we refer the reader to \cite{MR1720129, MR2519358, MR2102888, MR2073414, MR3034605}.

Another necessity of  $H^{\gamma}_{p,\theta}$-type spaces lies in  regularity theory of (deterministic) PDEs defined on non-smooth domains, say $C^1$ domains (see Remark \ref{regularity}). Since the boundary is not supposed be regular enough, we have to look for solutions in functions spaces with weights allowing the derivatives of solutions to blow up near the boundary. In the framework of H\"older spaces such a setting leads to investigating so-called intermediate Schauder estimates. Thus, even for deterministic problems, we require appropriate weights near the boundary to estimate the derivatives of solutions of PDEs. Then, it is natural that the coefficients of lower order terms are allowed to blow up near the boundary. See Assumptions \ref{assumption_div} and \ref{assumption_nondiv}. 

The remainder of the paper is organized as follows.
In Section \ref{setting}, we introduce spaces $\widetilde{\fH}_{p, \theta}^{\gamma}(S, T)$, $W_p^{\mathsf{s}/2, \mathsf{s}}((S, T) \times \partial\Omega)$, and state the main theorem.
In Section \ref{pre} we introduce and prove elementary properties of $\widetilde{\fH}_{p, \theta}^{\gamma}(S, T)$. In Sections \ref{trace} and \ref{exten}, by using integral representations we prove Propositions \ref{prop0712_01} and \ref{prop0715_01}, which are key ingredients in the proof of the main theorem.
Then we prove the main theorem in Section \ref{main}. Finally in Section \ref{app}, as an application of the main theorem, we solve divergence and non-divergence type parabolic equations with non-zero boundary value conditions in $\widetilde{\fH}_{p, \theta}^1(S, T)$ and $\widetilde{\fH}_{p, \theta}^2(S, T)$, respectively.

In this paper, $\bR^n$ stands for the $n$-dimensional Euclidean space and $x = (x_1, \ldots, x_n)$ denotes a point in $\bR^n$. We use the following standard notation
\[
\bR^{n}_+ = \{ (x_1, x'): x_1 > 0, x' = (x_2, \ldots, x_n) \in \bR^{n-1} \},
\]
\[
D_i = \partial / \partial x_i, \quad Du = u_x = (D_1u, \ldots, D_n u).
\]
For a multi-index $\alpha = (\alpha_1, \ldots, \alpha_n)$ where $\alpha_i$'s are non-negative integers, we denote
\[
D^{\alpha} = D_1^{\alpha_1} \ldots D_n^{\alpha_n}, \quad |\alpha| = \alpha_1 + \ldots + \alpha_n.
\]
For a domain $\Omega \subset \bR^n$, by $\mathcal{D}'(\Omega)$ we mean the space of all distributions on $\Omega$.

\section{Function spaces and Main result}
	\label{setting}

Let $\Omega$ be a domain in $\bR^n$ with non-empty boundary $\partial\Omega$.
We first recall the definition of $H_{p, \theta}^{\gamma}(\Omega)$ introduced in \cite{MR1708104} and \cite{MR1796011}.
For $k \in \bZ$ and a fixed integer $k_0 > 0$, 
define subsets $\Omega_k$ of $\Omega$ by
$$
\Omega_k = \{ x \in \Omega : e^{-k-k_0} < \rho(x) < e^{-k+k_0} \},
$$
where $ \rho(x) = \rho_{\Omega}(x) = \text{dist}(x, \partial \Omega)$.
Let $\{ \zeta_k, k \in \bZ \}$ be a collection of non-negative functions with the following properties:
\begin{equation}
							\label{eq0323_03}
\zeta_k \in C_0^{\infty}(\Omega_k),
\quad
| D^m \zeta_k (x) | \le N(m) e^{mk},
\quad 
\sum_{k \in \bZ} \zeta_k(x) = 1 \quad \textrm{on} \,\, \Omega.
\end{equation}
If $\Omega_k$ is an empty set, then the corresponding $\zeta_k$ is identically zero.

\begin{definition}
							\label{def0109_1}
Let $\theta$, $\gamma \in \bR$ and $1 < p < \infty$.
We set
$$
H_{p,\theta}^{\gamma}(\Omega)
:= \left\{ u \in \cD'(\Omega) :
\| u \|_{H_{p,\theta}^{\gamma}(\Omega)}^p
:= \sum_{k \in \bZ} e^{k \theta}  \| \zeta_{-k}(e^k \cdot) u(e^k \cdot) \|_{H_{p}^{\gamma}}^p < \infty \right\}
$$
and
$$
\widetilde{H}_{p,\theta}^{\gamma}(\Omega) := \left\{ u \in \cD'(\Omega) :   \|u\|_{\widetilde{H}_{p, \theta}^{\gamma}(\Omega)}^p := \| u \|_{H_{p, \theta}^{\gamma - 1}(\Omega)}^p + \|Du\|_{H_{p, \theta}^{\gamma - 1}(\Omega)}^p < \infty   \right\},
$$
where $\| \cdot \|_{H^{\gamma}_p}$ is the norm of 
Bessel potential space $H_{p}^{\gamma}(\bR^{d})$ ($= H_{p}^{\gamma}$) and $L_{p, \theta}(\Omega)$ ($= H_{p, \theta}^{0}(\Omega)$) is the space of functions summable to the power $p$ with respect to the measure $\rho(x)^{\theta - n} \, dx$, that is,
\[
f \in L_{p,\theta}(\Omega) \quad \iff \quad \int_\Omega |f(x)|^p \rho(x)^{\theta-n} \, dx < \infty.
\]
\end{definition}

\begin{remark}
If $\gamma \geq 1$ and $\Omega$ is bounded, then the norm of $\widetilde{H}_{p,\theta}^{\gamma}(\Omega)$ is equivalent to $\| u \|_{L_{p, \theta}(\Omega)} + \|Du\|_{H_{p, \theta}^{\gamma - 1}(\Omega)}$.
Indeed, if $\gamma = 1$, the equivalence is clear even if $\Omega$ is not bounded.
If $\gamma \in (k,k+1]$, where $k$ is a positive integer, then by using Proposition \ref{propertyweighted} (5) sufficiently many times, we get
\[
\|u\|_{\bH_{p,\theta}^{\gamma-1}(\Omega)} \simeq \|u\|_{\bH_{p,\theta}^{\gamma-(k+1)}(\Omega)} + \sum_{j=2}^{k+1}\|Du\|_{\bH_{p,\theta+p}^{\gamma-j}(\Omega)},
\]
where $\gamma - (k+1) \leq 0$ and $\gamma - j \leq \gamma - 2$ for $j=2,\ldots,k+1$.
Then, by Proposition \ref{propertyweighted} (2) and (6), the right-hand side of the above equivalence is bounded by a constant times $\| u \|_{L_{p, \theta}(\Omega)} + \|Du\|_{H_{p, \theta}^{\gamma - 2}(\Omega)}$.
This implies that 
\[
\|u\|_{\widetilde{H}_{p,\theta}^{\gamma}(\Omega)} \leq N \| u \|_{L_{p, \theta}(\Omega)} + N \|Du\|_{H_{p, \theta}^{\gamma - 1}(\Omega)}.
\]
The opposite inequality is clear by again Proposition \ref{propertyweighted} (2) because $\gamma \geq 1$.
\end{remark}

\begin{remark}
							\label{rem0108_1}
As in \cite{MR2111792}, one can find $\zeta_k$ as follows.
Let $\xi \in C_0^{\infty}(\bR_{+})$ be a function satisfying 
$$
\sum_{m=-\infty}^{\infty} \xi(e^{m+t}) > 0
$$
for all $t \in \bR$. Then, for $x \in \Omega$ and $k \in \bZ = \{ 0, \pm 1, \dots \}$,
define
$$
\zeta_k(x) = \left(\sum_{\ell =-\infty}^{\infty}\xi \left(e^\ell \psi(x) \right)\right)^{-1}\xi \left(e^k \psi(x) \right),
$$
where $\psi$ is a regularized distance introduced in \cite{MR0290095}.
In particular, $\psi$ satisfies
\begin{enumerate}
\item
$N_1 \rho(x) \le \psi(x) \le N_2(n) \rho(x)$, 
$x \in \Omega$
\item
$\psi(x)$ is $C^{\infty}$ in $\Omega$ and
$$
\left| D^{\alpha} \psi(x) \right| \le N(\alpha) \, \rho(x)^{1-|\alpha|},
$$
where $N_1$, $N_2(n)$, and $N(\alpha)$ are appropriate constants.
\end{enumerate}
If $\Omega$ is a bounded $C^1$ domain, then the definition of $H_{p,\theta}^\gamma(\Omega)$ in Definition \ref{def0109_1} is equivalent to that given in \cite{MR2111792}.
\end{remark}

We have the following properties of weighted Sobolev spaces $H_{p,\theta}^{\gamma}(\Omega)$.
For details, see \cite{MR1796011}.
Note that $\psi^\nu \cH_{p,\theta}^\gamma(\Omega)$ means the collection of $u$ such that $u = \psi^\nu v$ for some $v \in \cH_{p,\theta}^\gamma(\Omega)$, where $\nu \in \bR$ and $ 0 < \psi \in C^\infty(\Omega)$.

\begin{proposition}\label{propertyweighted}
Let $\theta$, $\gamma \in \bR$ and $1 < p < \infty$.
\begin{enumerate}

\item
$C_0^{\infty}(\Omega)$ is dense in $H_{p,\theta}^{\gamma}(\Omega)$.

\item
If $\gamma_1 > \gamma_2$, then 
$H_{p,\theta}^{\gamma_1}(\Omega) \subset 
H_{p,\theta}^{\gamma_2}(\Omega)$.

\item
If $\psi$ is a $C^{\infty}(\Omega)$ function so that 
$N_1 \rho(x) \le \psi(x) \le N_2 \rho(x)$, then
$\psi H_{p, \theta}^{\gamma}(\Omega) = H_{p, \theta - p}^{\gamma}(\Omega)$.

\item
If $\gamma$ is a non-negative integer, then 
$$
H_{p, \theta}^{\gamma}(\Omega)
= \{ u : \rho(x)^{|\alpha|} D^{\alpha}u \in L_{p, \theta}(\Omega), 0 \le |\alpha| \le \gamma \}.
$$

\item
$u \in H_{p, \theta}^{\gamma}(\Omega)$ iff
$u \in H_{p, \theta}^{\gamma-1}(\Omega)$ and
$Du \in H_{p, \theta+p}^{\gamma-1}(\Omega)$. 
In addition,
the norm $\| u \|_{H_{p, \theta}^{\gamma}(\Omega)}$ is equivalent to
$ \| u \|_{H_{p, \theta}^{\gamma-1}(\Omega)} + \| u_x \|_{H_{p, \theta+p}^{\gamma-1}(\Omega)}$.

\item
Let $\Omega$ be a bounded domain.
Then for $\theta_1 < \theta_2$, we have
$$
H_{p, \theta_1}^{\gamma}(\Omega) \subset H_{p, \theta_2}^{\gamma}(\Omega)
\quad \mbox{and} \quad
\| u \|_{H_{p, \theta_2}^{\gamma}(\Omega)}
\le N \| u \|_{H_{p, \theta_1}^{\gamma}(\Omega)},
$$
where $N$ is independent of $ u \in H_{p, \theta_1}^{\gamma}(\Omega)$.

\end{enumerate}
\end{proposition}

For $-\infty \le S <  T \le \infty$, denote
$$
\bH_{p,\theta}^\gamma(S,T) = L_p \left( (S,T), H_{p,\theta}^\gamma\left( \Omega \right) \right),
\quad
\bL_{p,\theta}(S,T) = L_p \left( (S,T), L_{p,\theta} \left( \Omega \right) \right)
$$
with the norms
\[
\|u\|_{\bH_{p,\theta}^\gamma(S,T)}^p = \int_S^T \|u(t,\cdot)\|_{H_{p,\theta}^\gamma(\Omega)}^p \, dt, \,\,
\|u\|_{\bL_{p,\theta}(S,T)}^p = \int_S^T \|u(t,\cdot)\|_{L_{p,\theta}(\Omega)}^p \, dt.
\]

\begin{definition}
	\label{sol_space}
Let $\gamma, \theta \in \bR$, and $1 < p < \infty$. We write $u \in \fH_{p,\theta}^\gamma(S,T)$ if
$$
u \in \bH_{p,\theta - p}^{\gamma}(S,T),
\quad
u_t \in \bH_{p,\theta+p}^{\gamma-2}(S,T),
$$
and write $u \in \widetilde{\fH}_{p,\theta}^\gamma(S,T)$ if
$$
u \in \bH_{p, \theta}^{\gamma - 1}(S,T),
\quad
Du \in \bH_{p,\theta}^{\gamma-1}(S,T),
\quad
u_t \in \bH_{p,\theta+p}^{\gamma-2}(S,T),
$$
where $u_t \in \cD'(\Omega)$ is defined by
$$
\left(u(t,\cdot), \phi \right) = \left(u(s,\cdot), \phi \right)  + \int_{s}^t \left( u_t(r,\cdot), \phi \right) \, d r
$$
for almost every $s,t \in [S,T]$.
The notation $(w, \phi)$ is to be interpreted as applying a test function $\phi \in C_0^\infty(\Omega)$ to a distribution $w$ in $\cD'(\Omega)$. 
For a non-negative integer $k$, we write $ u \in \mathcal{W}_{p, \theta}^{k+1}\left( S,T \right)$ if
\[
D^{\alpha}u, D^{\beta}u_t \in \bL_{p, \theta}(S,T)
\]
for any multi-indices (with respect to the spatial variables) $\alpha$ and $\beta$ such that $|\alpha| \le k+1$ and $|\beta| \le k-1$.
We set the norms of the these spaces as 
\[
\begin{aligned}
&\|u\|_{\fH_{p,\theta}^\gamma(S,T)} := \|u\|_{\bH_{p,\theta - p}^{\gamma}(S,T)} + \|u_t\|_{\bH_{p,\theta+p}^{\gamma-2}(S,T)},
\\
&\|u\|_{\widetilde{\fH}_{p,\theta}^\gamma(S,T)} := \| u \|_{\bH_{p, \theta}^{\gamma - 1}(S, T)} + \|Du\|_{\bH_{p,\theta}^{\gamma-1}(S,T)} + \|u_t\|_{\bH_{p,\theta+p}^{\gamma-2}(S,T)},
\\
&\left\| u \right\|_{\mathcal{W}_{p, \theta}^{k+1}\left( (S, T) \times \Omega \right)} := \sum_{|\alpha| \le k+1}\left\| D^{\alpha}u \right\|_{\bL_{p, \theta}\left( S, T  \right)} + \sum_{|\beta| \le k-1} \left\| D^{\beta}u_t \right\|_{\bL_{p, \theta}\left( S, T  \right)}.
\end{aligned}
\]
\end{definition}

The fact that $\widetilde{\fH}_{p,\theta}^{\gamma}(S,T)$ is a Banach space is proved in Lemma \ref{lem0728_1}.
In particular, if $\gamma=1$ and $(S, T) = \bR$, we have
$$
\widetilde{\fH}_{p,\theta}^1 (\bR) = \{ u : u, Du \in \bL_{p,\theta}(\bR), u_t \in \bH_{p,\theta+p}^{-1}(\bR) \}.
$$

For the remainder of this section, we assume that $\Omega = \bR^n_+$ or $\Omega$ is a bounded Lipschitz domain, i.e. $\partial\Omega \in C^{0, 1}$.
Note that if $v \in \bH_{p, \theta + p}^{-1}(S,T)$, there exist $g_i \in \bL_{p, \theta}(S,T)$, $i = 1, \ldots, n$, such that $v = \sum_i D_ig_i$ in $(S,T) \times \Omega$.
Also note that $\| v \|_{\bH_{p, \theta + p}^{-1}(S, T)}$ is equivalent to $\sum_i \| g_i \|_{\bL_{p, \theta}(S, T)}$.
Thus, if $u_t \in \bH_{p, \theta + p}^{-1}(S, T)$, there exist $g_i \in \bL_{p, \theta}(S,T)$, $i = 1, \ldots, n$, such that
$$
\int_{(S, T) \times \Omega} u  \phi_t \,dx\,dt = \int_{(S, T) \times \Omega} g_i D_i\phi \,dx\,dt
$$
for all $\phi \in C_0^{\infty}((S, T) \times \Omega)$.
For details about the statements regarding $\bH_{p, \theta + p}^{-1}(S,T)$ and the norm equivalence, see \cite[Remark 5.3]{MR1708104} for $\Omega = \bR^n_+$ and \cite[Theorem 3.2]{MR1796011} (with a diffeomorphism introduced in \cite{MR2352844}) for $\partial \Omega \in C^{0, 1}$. \label{page_ref0302_1}

To discuss boundary values of functions defined on $(S,T) \times \Omega$, we introduce function spaces defined on the lateral boundaries of cylindrical domains.
For $\Omega = \bR^n_+$, $1 \le p<\infty$, $0<\mathsf{s}<1$, and a function $g(t,x')$ defined on $(S,T) \times \bR^{n-1}=(S,T) \times \partial \Omega$, we write $g \in W_p^{\mathsf{s}/2,\mathsf{s}}\left((S,T) \times \bR^{n-1}\right)$ if
$$
\|g\|_{W^{\mathsf{s}/2,\mathsf{s}}_p\left((S,T) \times \bR^{n-1}\right)}=\|g\|_{L_p}+[g]_{W^{\mathsf{s}/2,0}_p}+[g]_{W^{0,\mathsf{s}}_p} < \infty,
$$
where 
$$
[g]_{W_p^{\mathsf{s}/2,0}}^p = [g]_{W_p^{\mathsf{s}/2,0}\left((S,T) \times \bR^{n-1}\right)}^p = \int_{S}^T \int_{S}^T \frac{\|g(r,\cdot) - g(t,\cdot)\|_{L_p(\bR^{n-1})}^p}{|t-s|^{1+\mathsf{s}p/2}} \, ds \, dt,
$$
$$
[g]_{W^{0,\mathsf{s}}_p}^p = [g]_{W^{0,\mathsf{s}}_p\left((S,T) \times \bR^{n-1}\right)}^p = \int_{S}^T \int_{\bR^{n-1} \times \bR^{n-1}} \frac{|g(t,x') - g(t,y')|^p}{|x'-y'|^{n-1+\mathsf{s}p}} \, dx' \, dy'\,dt.
$$
In particular, if $(S, T) = \bR$, we see that
\begin{equation}
							\label{eq0401_04}
\begin{aligned}
\left[g\right]_{W_p^{\mathsf{s}/2,0}(\bR \times \bR^{n-1})}^p &= \int_\bR |\tau|^{-(1+\mathsf{s}p/2)} \| g(\cdot + \tau, \cdot) - g(\cdot,\cdot)\|_{L_p(\bR \times \bR^{n-1})}^p \, d\tau
\\
&= 2 \int_0^\infty \tau^{-(1+\mathsf{s}p/2)} \| g(\cdot + \tau, \cdot) - g(\cdot,\cdot)\|_{L_p(\bR \times \bR^{n-1})}^p \, d\tau.
\end{aligned}
\end{equation}
When $n=1$, $g \in W_p^{\mathsf{s}/2, \mathsf{s}}\left( \left(S, T \right) \times \bR^{n-1} \right)$ obviously means that $g \in W_p^{\mathsf{s}/2}\left(S, T \right)$ with
$$
\left\| g \right\|_{L_p\left( \left(S, T \right) \times \bR^{n-1} \right)}^p = \int_S^T \left| g \right|^p \,dt,
$$
$$
\left[ g \right]_{W_p^{\mathsf{s}/2, \mathsf{s}} \left( \left(S, T \right) \times \bR^{n-1} \right)}^p = \left[ g \right]_{W_p^{\mathsf{s}/2, 0} \left(S, T \right) }^p  = \int_S^T \int_S^T  \frac{|g(t) - g(s)|^p}{|t - r|^{1+\mathsf{s}p/2}}  \, ds \,dt .
$$
For a bounded Lipschitz domain $\Omega$, we use a partition of unity to define $W_{p}^{\mathsf{s}/2, \mathsf{s}}(\left(S, T\right) \times \partial \Omega)$.
Note that the boundary of a Lipschitz domain is locally the graph of a Lipschitz continuous function.
Thus there exist a finite number of balls $B_{j}$, $j = 1, \cdots, M$, in $\bR^n$ with radius $R$ and Lipschitz continuous functions $h_j$
defined on $\bR^{n-1}$ such that
$$
\bigcup_{j = 1}^{M} B_{j} \supset \partial\Omega
\quad
\text{and}
\quad
\Omega \cap B_j =  U_j \cap B_j,
$$
where, after relabeling and re-orienting the coordinate axes if necessary, 
$U_j = \{(x_1, x'): x_1 > h_j(x')\}$.
There also exist one-to-one functions 
$\Psi_j$ from $U_j$ onto $\bR^n_+$ satisfying 
$\Psi_j^{-1}(0,y') = \partial U_j$ for $y' \in \bR^{n-1}$.
Then we find infinitely differentiable functions 
$\varphi_{j}(x)$ defined on $\bR^n$
such that $\text{supp}\varphi_{j} \subset B_{j}$
and $\sum_{j = 1}^{M}\varphi_{j}(x) = 1$ on $\partial \Omega$.
Fix a $K > 0$ such that $\| \Psi_j \|_{C^{0, 1}} \le K$ for $j = 1, \ldots, M$.

\begin{definition}\label{trace_space}
Let $1 \le p < \infty$, $-\infty \le S < T \le \infty$, $\partial \Omega \in C^{0,1}$ and $ 0 < \mathsf{s} < 1$.
We set
\begin{multline*}
W_{p}^{\mathsf{s}/2, \mathsf{s}}(\left( S, T \right) \times \partial \Omega)
= \{ f \in L_p(\left( S, T \right) \times \partial \Omega) :
\\
(\varphi_j f)(t, \Psi_j^{-1}(0, y')) \in 
W_{p}^{\mathsf{s}/2, \mathsf{s}}(\left( S, T \right) \times \bR^{d-1}), \,  j = 1, \cdots, M \},
\end{multline*}
$$
\| f \|_{W_{p}^{\mathsf{s}/2, \mathsf{s}}(\left( S, T \right) \times \partial \Omega)} = \sum_{j=1}^{M}
\| (\varphi_j f)(\cdot, \Psi_j^{-1}(0, \cdot)) \|_{W_{p}^{\mathsf{s}/2, \mathsf{s}}(\left( S, T \right) \times \bR^{d-1})}.
$$
\end{definition}

For given $g \in W_p^{\mathsf{s}/2, \mathsf{s}}((S, T) \times \partial\Omega)$, we write $g \in W_{p, 0}^{\mathsf{s}/2, \mathsf{s}}((S, T) \times \partial\Omega)$ if there exists a sequence 
\[
\{g_k\}_{k=1}^{\infty} \subset W_p^{\mathsf{s}/2, \mathsf{s}}((-\infty, T) \times \partial\Omega) \cap C(\overline{(-\infty, T) \times \partial\Omega})
\]
such that $g_k \to g$ in $W_p^{\mathsf{s}/2, \mathsf{s}}((S, T) \times \partial\Omega)$ as $k \to \infty$ and $g_k=0$ for $t \le S$.

\begin{remark}
The norm defined above for $W_{p}^{\mathsf{s}/2, \mathsf{s}}(\left( S, T \right) \times \partial \Omega)$ 
is independent of the choice of the balls $B_j$ as well as of
the choice of the functions $\Psi_j$ and $\varphi_j$.
For details regarding this, see \cite[Lemma 3.6.1]{MR503903}.
\end{remark}

If a function $u$ is continuous up to the boundary of the domain, $\cT u$, the lateral trace operator $\cT$ applied to $u$, is simply the restriction of $u$ on the lateral boundary of the domain, i.e., $\cT u = u|_{(S,T) \times \partial \Omega}$.
Then, as usual, we define the trace operator
$$
\cT: \widetilde{\fH}_{p,\theta}^\gamma(S,T) \to W_p^{\mathsf{s}/2,\mathsf{s}}\left( (S,T) \times \partial \Omega \right)
$$
using the denseness of $C_0^\infty(\overline{(S,T) \times \Omega})$ in $\widetilde{\fH}_{p,\theta}^\gamma(S,T)$ (see Lemma \ref{lem0603_1}),
provided that the operator is bounded.
Throughout the paper, by $C_0^\infty(\overline{(S,T) \times \Omega})$, we mean the collection of infinitely differentiable functions defined on $\overline{(S,T) \times \Omega}$ having support in an intersection of a ball in $\bR^{n+1}$ with $\overline{(S,T) \times \Omega}$.
In particular, $u$ may not be zero on the boundary of $(S,T) \times \Omega$ if $u \in C_0^\infty(\overline{(S,T) \times \Omega})$, whereas, as usual, $u$ is zero on the boundary if $u \in C_0^\infty((S,T) \times \Omega)$.

As the main result of this paper, we prove that the trace operator $\cT$ is bounded and has a right inverse, which can be called an extension operator.
Here is the statement, proved in Section \ref{main}.

\begin{theorem}[Trace theorem]
							\label{thm0716_01}
Let $\gamma \ge 1$, $1 < p < \infty$, $ T \in (-\infty, \infty]$, $\Omega = \bR^n_+$ or $\partial\Omega \in C^{0, 1}$, and $n-1< \theta < n-1+p$ with $\mathsf{s} = (p-\theta+n-1)/p$.
Then the operator
$$
\cT: \widetilde{\fH}_{p,\theta}^\gamma( -\infty, T) \to W_p^{\mathsf{s}/2,\mathsf{s}}\left( (-\infty,T) \times \partial\Omega \right)
$$
with $\cT u = u|_{(-\infty,T) \times \partial\Omega}$ for $u \in \widetilde{\fH}_{p,\theta}^\gamma(-\infty,T) \cap C(\overline{(-\infty,T) \times \Omega})$ is bounded and satisfies
\begin{equation}
	\label{eq_trace}
\| \cT u \|_{W_p^{\mathsf{s}/2,\mathsf{s}}\left( (-\infty,T) \times \partial\Omega \right)} \le N \|u\|_{\widetilde{\fH}_{p,\theta}^\gamma(-\infty,T)},
\end{equation}
where $N = N(n,p,\theta)$ for $\Omega = \bR^n_+$ and $N = N(n,p,\theta, M, K, R)$ for $\partial\Omega \in C^{0, 1}$.

Moreover, $\cT$ has a right inverse (an extension operator)
$$
\cS: W_p^{\mathsf{s}/2,\mathsf{s}}\left( (-\infty,T) \times \partial\Omega  \right) \to \widetilde{\fH}_{p,\theta}^\gamma(-\infty,T)
$$
such that
$\cT(\cS g) = g$ for $g \in W_p^{\mathsf{s}/2,\mathsf{s}}\left( (-\infty,T) \times \partial\Omega \right)$ and
\begin{equation}
			\label{extension}
\|\cS g\|_{\widetilde{\fH}_{p,\theta}^\gamma(-\infty,T)}
 \le N \|g\|_{W_p^{\mathsf{s}/2,\mathsf{s}}\left( (-\infty,T) \times \partial\Omega \right)},
\end{equation}
where $N=N(n,p,\theta,\gamma)$ for $\Omega = \bR^n_+$ and $N = N(n,p,\theta, \gamma, M, K, R)$ for $\partial\Omega \in C^{0, 1}$. 

\end{theorem}

\begin{remark}
	\label{finite_time}
By using an extension with respect to the $t$ variable and then multiplying a cut-off function, one can check that Theorem \ref{thm0716_01} holds for $u \in \widetilde{\fH}_{p, \theta}^{\gamma}(S, T)$ and $g \in W_{p}^{\mathsf{s}/2, \mathsf{s}}((S, T) \times \partial\Omega)$ when $-\infty < S < T < \infty$.
In this case, the constants $N$ of \eqref{eq_trace} and \eqref{extension} depend also on $T-S$.

Indeed, by scaling we may assume that $S=0$ and $T=1$.
For the trace part, it suffices to prove
\begin{equation}
							\label{eq0302_01}
\| \cT u \|_{W_p^{\mathsf{s}/2,\mathsf{s}}\left( (0,1) \times \partial\Omega \right)} \le N \|u\|_{\widetilde{\fH}_{p,\theta}^1(0,1)}
\end{equation}
because $\|u\|_{\widetilde{\fH}_{p,\theta}^1(0,1)} \le \| u \|_{\widetilde{\fH}_{p,\theta}^{\gamma}(0,1)}$ for any $\gamma \ge 1$.
To prove this inequality, we first consider $\Omega = \bR^n_+$.
Let $u \in \widetilde{\fH}_{p,\theta}^1(0,1)$.
Then, thanks to Theorem \ref{thm0716_01}, the inequality \eqref{eq0302_01} follows if there exists $v \in \widetilde{\fH}_{p, \theta}^1(\bR)$ such that
$\cT u = \cT v$ on $(0,1) \times \partial \bR^n_+$ and
\begin{equation}
							\label{eq0302_02}
\| v \|_{\widetilde{\fH}_{p, \theta}^1(\bR)} \le N \| u \|_{\widetilde{\fH}_{p, \theta}^1(0, 1)},
\end{equation}
where $N = N(n, p, \theta)$.
We here note that $u_t = D_i g_i$ in $(0,1) \times \bR^n_+$ for some $g_i \in \bL_{p,\theta}(0,1)$.
For the existence of such $g_i\in \bL_{p,\theta}(0,1)$, see the explanation about $\bH_{p, \theta + p}^{-1}$ on page \pageref{page_ref0302_1}.
To come up with $v$ satisfying the aforementioned properties, we take smooth functions $\zeta: \bR \to \bR_+$, $\eta: \bR_+ \to \bR_+$ such that $\zeta = 1$ on $(0, 1)$, $\operatorname{supp}\zeta \subset (-1, 2)$, and $\eta = 1$ on $(0, 1)$, $\eta = 0$ on $(2, \infty)$.
Then we extend $u \in \widetilde{\fH}_{p, \theta}^1(0, 1)$ to $\overline{u} \in \widetilde{\fH}_{p, \theta}^1(-1, 2)$ by using even extensions with respect to the time variable at $t = 0$ and $1$.
Observe that $\overline{u}_t = D_i\overline{g}_i$ in $(-1, 2) \times \bR^n_+$ where $\overline{g}_i$ is the odd extension of $g_i$ with respect to the time variable at $t=0$ and $t=1$ so that $\| \overline{g}_i \|_{\bL_{p, \theta}(-1, 2)}$ is comparable with $\| g_i \|_{\bL_{p, \theta}(0, 1)}$.
Then set
\[
v(t, x) := \overline{u}(t, x)\zeta(t)\eta(x_1).
\]
We see that $v, Dv \in \bL_{p,\theta}(\bR)$ and $\||v|+|Dv|\|_{\bL_{p,\theta}(\bR)}$ is comparable with $\||u|+|Du|\|_{\bL_{p,\theta}(0,1)}$.
We also see that
\[
v_t = D_i G_i
\]
in $\bR \times \bR^n_+$, where
$$
G_1(t, x) =  \int_{x_1}^{\infty}  \left[\overline{g}_1\left(t, y_1, x' \right) \zeta \left(t \right) \eta'\left(y_1\right) +  \overline{u}\left(t, y_1, x'\right) \zeta'\left(t\right) \eta\left(y_1\right) \right] \,dy_1
$$
$$
+ \overline{g}_1(t, x)\zeta(t) \eta(x_1)
$$
and
$$
G_i(t, x) = \overline{g}_i(t, x)\zeta(t) \eta(x_1)\quad \textrm{for} \quad i = 2, \ldots, n,
$$
with the inequalities
\begin{equation}
							\label{eq0302_03}
\|G_1\|_{\bL_{p, \theta}(\bR )} \le N(n, p, \theta) \left( \|g_1\|_{\bL_{p, \theta}(0, 1) } + \|u\|_{\bL_{p, \theta}(0, 1)}  \right),
\end{equation}
$$
\|G_i\|_{\bL_{p, \theta}(\bR )} \le N(n, p, \theta) \|g_i\|_{\bL_{p, \theta}(0, 1) } \quad \textrm{for} \quad i = 2, \ldots , n.
$$
In particular, the inequality \eqref{eq0302_03} follows from the observation that
$$
\int_0^{\infty} \left| \int_{x_1}^{\infty} \overline{g}_1(t, y_1, x') \zeta(t) \eta'(y_1) +  \overline{u}(t, y_1, x') \zeta'(t) \eta(y_1) \,dy_1 \right|^p  x_1^{\theta - n} \,dx_1
$$
$$
\le N \int_0^{\infty} \left| \overline{g}_1(t,x) \zeta(t) \eta'(x_1) +  \overline{u}(t,x) \zeta'(t) \eta(x_1) \right|^p  x_1^{p+\theta - n}\,dx_1
$$
$$
\le N2^p \int_0^{\infty} \left| \overline{g}_1(t,x) \zeta(t) +  \overline{u}(t,x) \zeta'(t) \right|^p  x_1^{\theta - n}\,dx_1,
$$
where we used Hardy's inequality with $n-1 < \theta$ in the first inequality, and the fact that $\eta(x_1) = 0$ for $x_1 \geq 2$ in the second inequality.
Upon recalling the equivalence norm of $\widetilde{\fH}_{p,\theta}^1 (\bR)$ on page \pageref{page_ref0302_1}, we conclude that $v \in \widetilde{\fH}_{p,\theta}^1(\bR)$ and the inequality \eqref{eq0302_02} holds.
By the construction of $v$, it is clear that $\cT u = \cT v$ on $(0,1) \times \partial \bR^n_+$.
Therefore, the inequality \eqref{eq0302_01} is proved.

When $\partial\Omega \in C^{0, 1}$, we repeat the same argument with an extension $v(t, x) := \overline{u}(t, x) \zeta(t)$. Observe that there is no $\eta(x_1)$ since $\Omega$ is bounded.

For the extension part, we extend $g \in W_{p}^{\mathsf{s}/2, \mathsf{s}}((S, T) \times \partial\Omega)$ to $h \in W_p^{\mathsf{s}/2, \mathsf{s}}(\bR \times \partial\Omega)$ similarly as we extend $u \in \widetilde{\fH}_{p, \theta}^1(S, T)$ to $v \in \widetilde{\fH}_{p, \theta}^1(\bR)$ so that $\| h \|_{W_p^{\mathsf{s}/2, \mathsf{s}}(\bR \times \partial\Omega)}$ is comparable with $\| g\|_{W_{p}^{\mathsf{s}/2, \mathsf{s}}((S, T) \times \partial\Omega)}$.
\end{remark}

\section{Preliminaries}
	\label{pre}
In this section, we present some properties of the functions spaces in Section \ref{setting}, which are needed in the proofs of our main results.

\begin{lemma}
							\label{lem0728_1}
Let $1 < p < \infty$ and $-\infty \le S < T \le \infty$.
For any $\gamma, \theta \in \bR$, $\widetilde{\fH}_{p,\theta}^\gamma(S,T)$ is a Banach space.
\end{lemma}

\begin{proof}
We first note that, since $H_{p,\theta}^\gamma(\Omega)$ is a Banach space for any $\gamma, \theta \in \bR$ (see \cite[Remark 1.2]{MR1708104}), $\bH_{p, \theta}^\gamma(S, T) = L_p\left( (S, T), H_{p, \theta}^\gamma \left( \Omega \right) \right)$ is also a Banach space for any $\gamma, \theta \in \bR$.

To prove the lemma, we only show the completeness.
Let $\{u^m\}$ be a Cauchy sequence in $\widetilde{\fH}_{p,\theta}^\gamma(S,T)$. From the definition of $\widetilde{\fH}_{p,\theta}^\gamma(S,T)$ and the fact that $\bH_{p, \theta}^\gamma(S, T)$ is a Banach space, there exist $u, v \in \bH_{p, \theta}^{\gamma - 1}(S, T)$ and $w \in \bH_{p, \theta + p}^{\gamma - 2}(S, T)$ such that
$$
u^m \to u, \quad Du^m \to v, \quad \textrm{and} \quad u^m_t \to w 
$$
in the corresponding spaces as $m \to \infty$.
Thus, it suffices to show 
\[
v = Du \quad \textrm{and} \quad w = u_t.
\]
To prove this, we first find a subsequence, again denoted by $\{u^m\}$, satisfying the following for a.e. $t, s \in (S,T)$.
\[
u^m(t, \cdot) \to u(t, \cdot), \quad Du^m(t, \cdot) \to v(t, \cdot) \quad \textrm{in} \quad H_{p, \theta}^{\gamma-1}(\Omega),
\]
\[ u^m_t(t, \cdot) \to w(t,\cdot) \quad \textrm{in}  \quad H_{p, \theta + p}^{\gamma - 2}(\Omega)
\]
as $m \to \infty$, and 
\begin{equation}
							\label{eq0322_01}
\left( u^m \left(t, \cdot \right),  \phi \right) = \left( u^m \left(s, \cdot \right),  \phi \right) + \int_s^t\left( u^m_t \left(r, \cdot \right),  \phi \right) \,dr
\end{equation}
for $\phi \in C_0^{\infty}(\Omega)$.
Then, by the properties of $H_{p,\theta}^\gamma(\Omega)$ (in particular, see \cite[Proposition 2.4]{MR1796011}), we have
\[
(u^m(t,\cdot), \phi) \to (u(t,\cdot), \phi), \quad (Du^m(t,\cdot), \phi) \to (v(t,\cdot), \phi)
\]
\[
(u_t^m(t,\cdot),\phi) \to (w(t,\cdot), \phi)
\]
as $m \to \infty$ for almost every $t \in (S,T)$ and $\phi \in C_0^\infty(\Omega)$.
We also have
\begin{equation}
							\label{eq0323_02}
|\left(u^m_t(r,\cdot) - w(r,\cdot), \phi\right)| \leq N \|u^m_t(r,\cdot)-w(r,\cdot)\|_{H_{p,\theta+p}^{\gamma-2}(\Omega)}
\end{equation}
for almost every $r \in (S,T)$, where $N$ depends on $\phi$.
We then see that $v = Du$ and $w = u_t$.
In particular, for the latter equality, the inequality \eqref{eq0323_02} with the fact that $u_t^m$ converges to $w$ in $\bH_{p, \theta + p}^{\gamma - 2}(S, T)$ shows that
\[
\int_s^t \left(u_t^m(r,\cdot)-w(t,\cdot), \phi\right) \, dr \to 0,
\]
which along with \eqref{eq0322_01} implies
\[
\left( u \left(t, \cdot \right),  \phi \right) = \left( u \left(s, \cdot \right),  \phi \right) + \int_s^t\left( w \left(r, \cdot \right),  \phi \right) \,dr
\]
for almost every $s, t \in (S,T)$ and $\phi \in C_0^{\infty}(\Omega)$.
The lemma is proved.
\end{proof}

\begin{lemma}
							\label{lem0901}
Let $1 < p < \infty$, $T \in (-\infty, \infty]$, $\Omega = \bR^n_+$ or $\partial\Omega \in C^{0, 1}$, and $n-1 < \theta <n-1+p$.
Then
$C_0^\infty(\overline{(S,T) \times \Omega})$ is dense in $\widetilde{\fH}_{p,\theta}^1(S,T)$.
\end{lemma}

\begin{proof}
As noted in Remark \ref{finite_time}, it is enough to assume $(S, T) = \bR$. Moreover, using cut-off functions with respect to $(t,x)$, we may assume that $u$ vanishes for large $|t|+|x|$ when $\Omega = \bR^n_+$.

Case 1: $\Omega = \bR^n_+$.
We have $u, Du \in \bL_{p, \theta}(\bR)$ and $u_t = D_ig_i$ in $\bR \times \bR^n_+$ with $g = \left( g_1, \ldots, g_n  \right) \in \bL_{p, \theta}(\bR)^n$.
For  a function  $\eta \in C_0^{\infty}(\bR^{n+1})$ such that $\eta \ge 0$, $\int_{\bR^{n+1}} \eta \,dx \,dt = 1$, and
\[
\operatorname{supp} \eta \subset \{(t,x): -1<t<1, |x| < 1, x_1 < 0\},
\]
take
\[
u^{(\varepsilon)}(t,x) = \int_{\bR \times \bR^n_+} \eta_\varepsilon(t-s,x-y) u(s,y) \, ds \, dy,
\]
where $\eta_\varepsilon(t,x) = \varepsilon^{-n-2}\eta(t/\varepsilon^2,x/\varepsilon)$ and $\varepsilon > 0$.
Then, by the standard properties of mollifications together with the fact that $\eta(t,x) = 0$ for $x_1 \geq0$, one can see that $u^{(\varepsilon)} \in C_0^{\infty}(\overline{\bR \times \bR^n_+})$, $Du^{(\varepsilon)} = \left( Du \right)^{(\varepsilon)}$, and $u^{\left( \varepsilon \right)}_t = D_i g_i^{\left(  \varepsilon \right)}$ in $\bR \times \bR^n_+$ for each $\varepsilon > 0$.
Now observe that for any $f \in \bL_{p, \theta}(\bR) = L_p(\bR, L_{p,\theta}(\bR^n_+))$,
$$
\left| f^{\left( \varepsilon \right)} - f \right| \le N(n) \left( \mathcal{M}\tilde{f} + \left| f \right| \right)
$$
and $\mathcal{M}\tilde{f} + \left| f \right| \in \bL_{p, \theta}(\bR)$ by Hardy-Littlewood theorem for $A_p$-weight ($x_1^{\theta - n}$ is an $A_p$-weight), where $\tilde{f}$ is the zero extension of $f$ to the whole space $\bR \times \bR^n$ and $\cM \tilde{f}$ is the maximal function of $\tilde{f}$ based on the standard parabolic cubes.
Then by Lebesgue dominated convergence theorem, it follows that $\| f^{\left( \varepsilon \right)} - f \|_{\bL_{p, \theta}(\bR)} \to 0$ as $\varepsilon \to 0$. Thus $\|u - u^{\left( \varepsilon \right)}\|_{\widetilde{\fH}_{p,\theta}^1(\bR)} \to 0$ as $\varepsilon \to 0$.

Case 2: $\Omega$ is a bounded Lipschitz domain, i.e. $\partial\Omega \in C^{0,1}$.
In this case, we use a partition of unity argument.
Since $\partial\Omega$ satisfies the uniform exterior cone condition, we use this property to take a proper mollification  for $u \varphi_i \in \widetilde{\fH}_{p, \theta}^1(\bR)$ where $\{\varphi_i\}_{i=1}^M$ is a partition of unity.
The remainder is the same as the case for $\Omega = \bR^n_+$.
\end{proof}

The following lemmas will be used to prove the denseness of smooth functions in $\widetilde{\mathfrak{H}}_{p, \theta}^{\gamma}(S, T)$($\gamma \ge 1$). See Lemma \ref{lem0603_1}.

\begin{lemma}
							\label{lem0905_1}
Let $p \in [1,\infty)$, $-\infty \le S < T \le \infty$, $\Omega = \bR^n_+$ or $\partial\Omega \in C^{0, 1}$, $-1 < \theta - n$, and $k$ be a positive integer.
Then $\mathcal{W}_{p, \theta + kp}^{k+1}\left(S, T\right)$ is a Banach space and $C_0^\infty(\overline{(S, T)\times \Omega})$ is dense in $\mathcal{W}_{p, \theta + kp}^{k+1}\left(S, T\right)$.
\end{lemma}

\begin{proof}
It is easily seen that $\mathcal{W}_{p, \theta + kp}^{k+1}\left(S, T\right)$ is a Banach space.

To prove the second assertion, note  that the weight is $\rho(x)^{\theta - n + kp}$, where $\theta - n + kp \ge 0$ for any positive integer $k$.
Thus, one can prove the assertion by following the proof of \cite[Theorem 7.2]{MR802206}.

Indeed, for instance, if $\Omega = \bR^n_+$ and $u \in \mathcal{W}_{p, \theta + kp}^{k+1}\left( S, T \right)$, we set
$$
u^{\lambda}(t, x) := u(t, x_1 + \lambda, x') \in \mathcal{W}_{p, \theta + kp}^{k+1}\left( S, T \right)
$$
for $\lambda > 0$. Since 
$D^{\alpha}u^{\lambda}(t, x) = (D^{\alpha}u) \left( t, x_1 + \lambda, x' \right)$ and $D^{\beta}u^{\lambda}_t(t, x) = (D^{\beta}u_t) \left( t, x_1 + \lambda, x' \right)$ in $(S, T) \times \bR^n_+$ for any multi-indices $\alpha$ and $\beta$ such that $|\alpha| \le k+1$ and $|\beta| \le k-1$, by following the proof of \cite[Theorem 7.2]{MR802206} word for word, we obtain the desired result.
\end{proof}

\begin{lemma}
							\label{lem0910_01}
Let $1 < p < \infty$, $-\infty \le S < T \le \infty$, and $\Omega = \bR^n_+$ or $\partial\Omega \in C^{0, 1}$.
Assume that $n - 1 < \theta$ and $k$ is a positive integer. Then	
$$
\mathcal{W}_{\theta + kp}^{k+1}(S, T ) \subset \widetilde{\fH}_{p, \theta}^k(S, T).
$$
\end{lemma}

\begin{proof}
Recall that $u \in \widetilde{\fH}_{p, \theta}^k(S, T)$ means $u \in \bH_{p, \theta}^{k-1}(S, T)$, $Du \in \bH_{p, \theta}^{k - 1}(S, T)$, and $u_t \in\bH_{p, \theta + p}^{k - 2}(S, T)$.

Case 1: $\Omega = \bR^n_+$.
By Hardy's inequality with the denseness of $C_0^{\infty}(\overline{(S, T) \times \Omega})$ in $\mathcal{W}_{\theta + kp}^{k+1}(S, T)$, we have
$$
\int_{S}^T \int_{\Omega} |D^{\alpha}u|^p x_1^{\theta - n + |\alpha|p} \,dx\,dt \le N \int_{S}^T \int_{\Omega} |D^{k}u|^p x_1^{\theta - n + kp} \,dx\,dt,
$$
$$
\int_{S}^T \int_{\Omega} |D^{\alpha}Du|^p x_1^{\theta - n + | \alpha |p} \,dx\,dt \le N \int_{S}^T \int_{\Omega} |D^{k+1}u|^p x_1^{\theta - n + kp} \,dx\,dt
$$
for any multi-index $\alpha$ such that $0 \le | \alpha | \le k$, and if $k \ge 2$,
$$
\int_{S}^T \int_{\Omega} |D^{  \beta} u_t |^p x_1^{\theta - n + \left(|  \beta | + 1 \right)p} \,dx\,dt \le N \int_{S}^T \int_{\Omega} |D^{k - 1}u_t |^p x_1^{\theta - n + kp} \,dx\,dt
$$
for any multi-index $\beta$ such that $0 \le |\beta| \le k-2$.
The case $k=1$ directly follows from the inequality
\[
\| u_t \|_{\bH_{p, \theta + p}^{-1}(S, T)} \le  \| u_t\|_{\bL_{p, \theta + p}(S, T)}.
\]

Case 2: $\Omega$ is a bounded Lipschitz domain, i.e. $\partial \Omega \in C^{0, 1}$.
Note that 
\[
u(t,\cdot) \in W_{p,\theta+kp}^{k+1}(\Omega), \quad u_t(t, \cdot) \in W_{p, \theta + kp}^{k-1}(\Omega)
\]
for a.e. $t \in (S,T)$, where
\[
W_{p, \vartheta}^\ell(\Omega) := \{ v \in \mathcal{D}'(\Omega) : D^{\alpha}v \in L_{p, \vartheta}(\Omega), 0 \le |\alpha| \le \ell \}.
\]
Then by the (Hardy type inequality) embedding theorem in \cite[Theorem 8.4 and Remark 8.8]{MR802206}, we again have the above three inequalities.
The lemma is proved.
\end{proof}
By following the proof of \cite[Lemma 2.17]{MR2386392}, we obtain following result.
Indeed, the lemma holds for $\gamma = 1$ by Lemma \ref{lem0901}.

\begin{lemma}
							\label{lem0911}
Let $\gamma, \theta \in \bR$, $1 < p < \infty$, $-\infty \le S < T \le \infty$, and $\Omega = \bR^n_+$.
For any $u \in \widetilde{\fH}_{p,\theta}^\gamma(S,T)$, there is a sequence $\{ u_k \}$ in $\widetilde{\fH}_{p,\theta}^\gamma(S,T)$ such that $u_k = 0$ for large values of $x_1$ and $u_k \to u$ in $\widetilde{\fH}_{p,\theta}^\gamma(S,T)$ as $k \to \infty$.
\end{lemma}

\section{Trace operator}
	\label{trace}

In this section, except Corollary \ref{trace_optimal}, we consider $u(t,x_1,x')$ defined on $\bR \times \bR^n_+$, that is, $\Omega = \bR^n_+$.

\begin{proposition}
							\label{prop0712_01}
Let $1 < p < \infty$, $\theta \in (n-1,n-1+p)$, and $\mathsf{s} =(n-1+p-\theta)/p$.
Then the trace operator $\cT$ from
$$
\widetilde{\fH}_{p,\theta}^1 (\bR) = \{ u \in \bL_{p,\theta}(\bR): Du \in \bL_{p,\theta}(\bR), u_t \in \bH_{p,\theta+p}^{-1}(\bR) \}
$$
to $W_p^{\mathsf{s}/2, \mathsf{s}}(\bR \times \bR^{n-1})$
is bounded.
Precisely,
\begin{equation}
	\label{eq1001}
\| u(\cdot, 0,\cdot)\|_{W_p^{\mathsf{s}/2,\mathsf{s}}(\bR \times \bR^{n-1})}
\le N \|u\|_{\widetilde{\fH}_{p,\theta}^1 (\bR)},
\end{equation}
provided that $u \in \widetilde{\fH}_{p,\theta}^1 (\bR) \cap C( \overline{\bR \times \bR^{n-1}} )$, where $N=N(n, p, \theta)$.
\end{proposition}

The proof of Proposition \ref{prop0712_01} is given at the end of this section.
Using Proposition \ref{prop0712_01} (also see Lemma \ref{lem0901}) and a flattening argument for weighted Sobolev spaces on Lipschitz domains, one can prove the following corollary.
See \cite[Lemma 2.14 and 2.15]{MR2352844} for details.

\begin{corollary}
							\label{trace_optimal}
Let $p$, $\theta$, and $\mathsf{s}$ be as in Proposition \ref{prop0712_01} and $\partial\Omega \in C^{0, 1}$.
Then the operator
$$
\cT: \widetilde{\fH}_{p,\theta}^1(\bR) \to W_p^{\mathsf{s}/2,\mathsf{s}}\left( \bR \times \partial\Omega \right)
$$
with $\cT u = u|_{\bR \times \partial\Omega}$ for $u \in \widetilde{\fH}_{p,\theta}^1(\bR) \cap C(\overline{ \bR \times \Omega})$ is bounded and satisfies
\begin{equation}
	\label{eq1002}
\| \cT u \|_{W_p^{\mathsf{s}/2,\mathsf{s}}\left( \bR \times \partial\Omega \right)} \le N \|u\|_{\widetilde{\fH}_{p,\theta}^1(\bR)},
\end{equation}
where $N = N(n,p,\theta, M, K, R)$.

\end{corollary}

To prove Proposition \ref{prop0712_01}, we start with the following lemma.

\begin{lemma}
							\label{lem0603_2}
Let $p$, $\theta$, and $\mathsf{s}$ be as in Proposition \ref{prop0712_01}.
For $u$ satisfying $u, D_i u \in \bL_{p, \theta}(\bR)$, $i = 1, \ldots, n$, we have
\begin{multline}
							\label{eq0222_01}
\|u(\cdot, 0, \cdot)\|_{L_p(\bR \times \bR^{n-1})}
+ [u(\cdot,0,\cdot)]_{W_p^{0,\mathsf{s}}(\bR \times \bR^{n-1})}
\\
\le N \left(\|u\|_{\bL_{p, \theta}(\bR)} + \|Du\|_{\bL_{p, \theta}(\bR)}\right),
\end{multline}
where $N = N(n, p, \theta)$.
\end{lemma}

\begin{proof}
In fact, one can obtain the lemma immediately by applying the trace theorem \cite[Theorem 2.8]{MR2386392} to $u(t,x)$ for each $t \in \bR$.
Since the trace theorem in \cite{MR2386392} is based on the interpolation theory on Banach spaces (see, for instance, \cite{MR503903}), for the reader's convenience, we here give an elementary and self-contained proof.

It is sufficient to prove the estimate \eqref{eq0222_01} for $u \in C_0^{\infty}(\overline{\bR \times \bR^n_+})$.
We first consider $n \ge 2$.
For $h > 0$ and $i = 2, \ldots, n$, we set
$$
\Delta_i^h u := u(t, x + he_i) - u(t, x),
$$
where $\{e_i\}_{i=1}^n$ is the standard basis of $\bR^n$. 
Observe that for each $t \in \bR$ and $i = 2, \ldots, n$,
\begin{align*}
| \Delta_i^h & u(t, 0, x') | \le \left| \Delta_i^h u(t, x_1, x') \right| + \left| \Delta_i^h u(t, x_1, x') -  \Delta_i^h u(t, 0, x') \right|
\\
&\le \left| \Delta_i^h u(t, x_1, x') \right| + \left|  u(t, x_1, x' + he_i) -  u(t, 0, x' + he_i) \right|
\\
&+ \left| u(t, x_1, x') - u(t, 0, x') \right|.
\end{align*}
By integrating both sides with respect to $x_1 \in (0, h)$ and dividing by $h$, we have
\begin{align*}
\left| \Delta_i^h u(t, 0, x') \right| &\le \frac{1}{h} \int_0^h \int_0^h \left| D_iu(t, x_1, x' + \ell e_i)  \right| \,d\ell \,dx_1
\\
&+ \int_0^h \left| D_1u(t, x_1, x' + he_i ) \right| + \left| D_1u(t, x_1, x')  \right| \,dx_1.
\end{align*}
Take $\alpha \in \bR$ so that $1-\mathsf{s}-1/p < \alpha < 1 - 1/p$. Then by Minkowski's inequality and H\"{o}lder's inequality, we have
$$
\left\| \Delta_i^h u(t, 0, \cdot)  \right\|_{L_p(\bR^{n-1})} \le N \int_0^h \left\| Du(t, x_1, \cdot) \right\|_{L_p(\bR^{n-1})} \,dx_1
$$
$$
\le N \left( \int_0^h \left\| Du(t, x_1, \cdot) \right\|_{L_p(\bR^{n-1})}^p x_1^{p\alpha} \,dx_1  \right)^{1/p} \left( \int_0^h x_1^{-q\alpha} \,dx_1\right)^{1/q},
$$
where $\alpha < 1/q := 1-1/p$.
This gives that
$$
\left\| \Delta_i^h u(t, 0, \cdot)  \right\|_{L_p(\bR^{n-1})}^p \le N h^{p-1-p\alpha} \int_0^h \left\| Du(t, x_1, \cdot) \right\|_{L_p(\bR^{n-1})}^p x_1^{p\alpha} \,dx_1.
$$
Therefore, with the fact that $1-\mathsf{s}-1/p < \alpha$, we obtain
\begin{equation}
							\label{eq0325_01}
\begin{aligned}
&\int_0^{\infty}  \left\| \Delta_i^h u(t, 0, \cdot)  \right\|_p^p h^{-1 - \mathsf{s}p}\,dh 
\\
&\le N \int_0^{\infty} \int_0^h \left\| Du(t, x_1, \cdot) \right\|_p^p x_1^{p\alpha} \,dx_1 h^{-p\alpha + p -\mathsf{s}p -2} \,dh
\\
&= N \int_0^{\infty} \int_{x_1}^{\infty} h^{-p\alpha + p -\mathsf{s}p -2} \, dh \left\| Du(t, x_1, \cdot) \right\|_p^p x_1^{p\alpha} \,dx_1
\\
&\le N \int_0^{\infty} \left\| Du(t, x_1, \cdot) \right\|_p^p x_1^{p -\mathsf{s}p -1} \,dx_1 = N \|Du(t, \cdot)\|_{L_{p, \theta}(\bR^n_+)}^p
\end{aligned}
\end{equation}
for $i = 2, \ldots, n$, where $N=N(n, p, \theta)$ and $\|\cdot\|_p = \|\cdot\|_{L_p(\bR^{n-1})}$.

For the estimate of $\| u(t, 0, \cdot ) \|_{L_p(\bR^{n-1})}$, note that
$$
\left| u(t, 0, x')  \right| \le \left| u(t, x_1, x') - u(t, 0, x')  \right| + \left| u(t, x_1, x')  \right|
$$
and by integrating both sides with respect to $x_1 \in (0, 1)$,
\begin{align*}
&\left| u(t, 0, x') \right|
\\
&\le \int_0^1 \int_0^{x_1} \left| D_1u(t, y_1, x')   \right|\,dy_1 \,dx_1 + \int_0^1 \left| u(t, x_1, x')  \right| \,dx_1
\\
&\le  \int_0^1  \left| D_1u(t, x_1, x')   \right| \,dx_1 + \int_0^1 \left| u(t, x_1, x')  \right| \,dx_1.
\end{align*}
Then, as above, by Minkowski's inequality and H\"older's inequality,
\begin{equation}
							\label{eq0325_02}
\begin{aligned}
&\left\| u(t, 0, \cdot) \right\|_{L_p(\bR^{n-1})} \le \int_0^1 \left\| u(t, x_1, \cdot) \right\|_{W_p^1
} x_1^{1 - \mathsf{s} - 1/p}x_1^{-1 + \mathsf{s} + 1/p} \,dx_1
\\
&\le \left( \int_0^1 \left\| u(t, x_1, \cdot) \right\|_{W_p^1
}^p x_1^{p - \mathsf{s}p - 1} \,dx_1 \right)^{1/p} \left( \int_0^1 x_1^{-1 + q\mathsf{s}} \, dx_1\right)^{1/q}
\\
&\le  N \||u(t,\cdot)| + |Du(t,\cdot)|\|_{L_{p,\theta}(\bR^n_+)},
\end{aligned}
\end{equation}
where $N = N(n, p, \theta)$ and $W_p^1 = W_p^1(\bR^{n-1})$.

Hence, by the estimates \eqref{eq0325_01} and \eqref{eq0325_02} as well as by the equivalence of the two semi-norms $[g]_{W_p^{\mathsf{s}}(\bR^{n-1})}$ and $[g]_{B_{p,p}^{\mathsf{s}}(\bR^{n-1})}$, where
\[
\left[  g \right]_{W_p^{\mathsf{s}}\left( \bR^{n-1} \right)}^p = \int_{\bR^{n-1}}\int_{\bR^{n-1}}  \frac{\left| g(x') - g(y') \right|^p }{ \left| x' - y'  \right|^{d-1 + \mathsf{s}p}} \,dx' \,dy'
\]
and 
\[
\left[ g \right]_{B_{p, p}^{\mathsf{s}}(\bR^{n-1})}^p = \sum_{i = 1}^{n-1}\int_0^{\infty}  \left\| \Delta_i^h g \right\|_{L_p\left( \bR^{n-1} \right)}^p |h|^{-1 - \mathsf{s}p} \,dh,
\]
for each $t \in \bR$, we have
$$
\| u(t,\cdot)\|_{W_p^{\mathsf{s}}(\bR^{n-1})} \le N \left(\|u(t,\cdot)\|_{L_{p,\theta}(\bR^n_+)} + \|Du(t,\cdot)\|_{L_{p,\theta}(\bR^n_+)}\right),
$$
where $N=N(n, p, \theta)$.  It only remains to integrate the $p$-power of both sides of the above inequality with respect to $t$ from $-\infty$ to $\infty$.

Finally, when $n=1$, the inequality \eqref{eq0222_01} directly follows from
$$
\left| u(t, 0)  \right| \le  \int_0^1  \left| Du(t,x)   \right| \,d x + \int_0^1 \left| u(t,x)  \right| \,d x
$$
because the semi-norm on the right-hand side of the inequality disappears.
The lemma is proved.
\end{proof}

In the lemma below we estimate the trace of a function by expressing the function as an integral form.
This type of argument is used in \cite{MR1257932}, where kernels are used for integral representations (see \cite{MR0152793}), to prove the boundedness of trace operators for functions in Sobolev spaces without weights.
Here we use  mollifications instead of kernels in the integral form.
\\

To estimate the seminorm $\left[ u\left( \cdot, 0, \cdot \right) \right]_{W_p^{s/2, 0}(\bR \times \bR^{n-1})}$, we start from an integral representation of $u \in C_0^{\infty}(\overline{\bR \times \bR^n_+})$ with $u_t = D_i g_i$ in $\bR \times \bR^n_+$.
Take $\phi(t, x) := \eta(x_1) \zeta(t, x')$ where $\eta$ and $\zeta$ are non-negative smooth functions such that 
$$
\operatorname{supp}\eta \subset (-1, -1/2), \quad \operatorname{supp}\zeta \subset (0, 1) \times B_1^\prime(0) \subset \bR \times \bR^{n-1},
$$
and $\|\phi\|_{L_1(\bR \times \bR^n)} = 1$.
Fix an $\varepsilon > 0$ and $ \phi_{\varepsilon}(t, x) := \varepsilon^{-n-2} \phi(t\varepsilon^{-2}, x\varepsilon^{-1})$.
Then for 
\[
u^{(\varepsilon)}(t,x) = \int_{\bR \times \bR^n_+} \phi_\varepsilon(t-s,x-y) u(s,y)\, dy \, ds,
\]
we write
$$
u(t, x) = u^{(\varepsilon)}(t, x) - \left( u^{(\varepsilon)}(t, x) - u(t, x) \right)
$$
$$
 = u^{(\varepsilon)}(t, x) - \int_{\bR \times \bR^{n}_+} \left( u(s, y) - u(t, x)\right) \phi_{\varepsilon}(t-s, x-y) dy\,ds
$$
$$
:= u^{(\varepsilon)}(t, x) - v(t, x; \varepsilon).
$$
Let $\gamma(\lambda)$ be a path from $(t, x) \in \bR \times \bR^n_+$ to $(s, y) \in \bR \times \bR^n_+$ defined by
$$
\gamma(\lambda) = \left( (1-\lambda)t + \lambda s, (1-\lambda^{1/2})x + \lambda^{1/2}y \right), \quad \lambda \in [0,1].
$$
Since
$$
u(s, y) - u(t, x) = \int_0^1 (\nabla_{t, x} u)(\gamma(\lambda)) \cdot  \gamma'(\lambda) \, d\lambda,
$$
we see that 
\[
v(t, x;\varepsilon) := v(t, x)
\]
$$
=\int_{\bR \times \bR^{n}_+}\int_0^1 \frac{1}{2} \lambda^{-\frac{1}{2}}\, \nabla u(\gamma(\lambda)) \cdot (y - x) \, d\lambda \,  \phi_{\varepsilon}(t-s, x-y) \,dy\,ds
$$
$$
+ \int_{\bR \times \bR^{n}_+}\int_0^1 u_t(\gamma(\lambda)) (s-t) \, d\lambda \,  \phi_{\varepsilon}(t-s, x-y) \,dy\,ds =:J_1 + J_2,
$$
where, by the change of variables $\gamma(\lambda) \to (l,z)$ and $\varepsilon^2 \lambda \to \lambda$,
\[
J_1 = \int_0^{\varepsilon^2} \frac{\lambda^{-\frac{n}{2} - \frac{3}{2}}}{2}\int_{\bR \times \bR^{n}_+} \nabla u(l, z) \cdot  \frac{z - x}{\sqrt{\lambda}}  \,  \phi \left( \frac{t-l}{\lambda}, \frac{x-z}{\sqrt{\lambda}} \right) dz\,dl \, d\lambda,
\]
\[
J_2 = \int_0^{\varepsilon^2}  \lambda^{-\frac{n}{2}-1} \int_{\bR \times \bR^{n}_+} u_t(l, z) \frac{l-t}{\lambda} \,  \phi \left( \frac{t-l}{\lambda}, \frac{x-z}{\sqrt{\lambda}} \right) dz\,dl \, d\lambda.
\]
For $J_2$, using the fact that $u_t = D_i g_i$ in $\bR \times \bR^n_+$ and the test function
\[
\phi \left(\frac{t - \cdot}{\lambda}, \frac{x - \cdot}{\sqrt{\lambda}}   \right) \in C_0^{\infty}(\bR \times \bR^n_+),
\]
we further have
\[
J_2 = \int_0^{\varepsilon^2}  \lambda^{-\frac{n}{2}-\frac{3}{2}} \int_{\bR \times \bR^{n}_+} \frac{l-t}{\lambda} g(l, z) \cdot  \nabla \phi \left( \frac{t-l}{\lambda}, \frac{x-z}{\sqrt{\lambda}} \right) \,dz\,dl \, d\lambda,
\]
where $g= (g_1,\ldots,g_n)$.
Thus, one can write
\begin{equation}
							\label{eq0326_01}
\begin{aligned}
u(t, 0, x') &= u^{(\varepsilon)}(t, 0, x') - v(t, 0, x')
\\
&= u^{(\varepsilon)}(t,0,x') - J_1(t,0,x') - J_2(t,0,x')
\\
&= u^{(\varepsilon)}(t,0,x') - \sum_{j=1}^3 \int_0^{\varepsilon^2} V_j(t,x',\lambda;\varepsilon)\, d\lambda,
\end{aligned}
\end{equation}
where
\begin{equation}
							\label{eq0329_01}
\begin{aligned}
V_1 &:= \frac{\lambda^{-\frac{n}{2} - \frac{3}{2}}}{2}\int_{\bR \times \bR^{n}_+} D_1 u(l, z) \frac{z_1}{\sqrt{\lambda}}  \,  \phi (\cdot) \, dz\,dl,
\\
V_2 &:= \frac{\lambda^{-\frac{n}{2} - \frac{3}{2}}}{2}\int_{\bR \times \bR^{n}_+} \nabla_{z'} u(l, z) \cdot  \frac{z' - x'}{\sqrt{\lambda}}  \,  \phi(\cdot)\,dz\,dl,
\\
V_3 &:= \lambda^{-\frac{n}{2}-\frac{3}{2}} \int_{\bR \times \bR^{n}_+} \frac{l-t}{\lambda} g(l, z) \cdot  \nabla \phi(\cdot)\,dz\,dl.
\end{aligned}
\end{equation}
Here,
\[
(\phi, \nabla \phi)(\cdot) = (\phi, \nabla \phi) \left( \frac{t-l}{\lambda}, \frac{-z_1}{\sqrt{\lambda}}, \frac{x'-z'}{\sqrt{\lambda}} \right).
\]

The lemma below is a key estimate for our trace theorem.

\begin{lemma}
							\label{lem0712_01}
Let $p$, $\theta$ and $\mathsf{s}$ be as in Proposition \ref{prop0712_01}.
Let $u \in C_0^{\infty}(\overline{\bR \times \bR^n_+})$ satisfy
$u_t = D_i g_i$ in $\bR \times \bR^n_+$, where $g_i \in \bL_{p,\theta}(\bR)$.
Then
\begin{equation}
							\label{eq0405_01}
\begin{aligned}
&\int_0^T \tau^{-1- \mathsf{s}p/2} \| u(\cdot + \tau, 0, \cdot) - u(\cdot,0,\cdot)\|_{L_p(\bR \times \bR^{n-1})}^p \, d \tau
\\
&\le N \left(\frac{T}{\varepsilon^2}\right)^{p-\mathsf{s}p/2} \int_0^{\varepsilon} \|g_i (\cdot,y_1,\cdot)\|_{L_p}^p \, y_1^{\theta - n} \, dy_1
\\
&+ N \int_0^{\sqrt{T}} F(y_1) \, y_1^{\theta - n} \, dy_1
+ N \int_0^{\varepsilon} F(y_1) \, y_1^{\theta - n} \, dy_1
\end{aligned}
\end{equation}
for any $\varepsilon, T \in (0,\infty)$,
where $N = N(n, p, \theta)$ and
$$
F(y_1) = \sum_{i=1}^n \| \left( \left| D_iu \right| +\left| g_i \right| \right) (\cdot,y_1,\cdot)\|_{L_p(\bR \times \bR^{n-1})}^p.
$$
\end{lemma}

\begin{proof}
Throughout the proof, we denote $\|\cdot\|_p = \|\cdot\|_{L_p(\bR \times \bR^{n-1})}$, where $\bR \times \bR^{n-1} = \{(t,x'): t \in \bR, x' \in \bR^{n-1}\}$.
By \eqref{eq0326_01}, for $\tau \in (0,T)$, we have
\[
|u(t+\tau,0,x') - u(t,0,x')| \leq |u^{(\varepsilon)}(t+\tau,0,x') - u^{(\varepsilon)}(t,0,x')|
\]
\[
+ \sum_{j=1}^3 \int_0^{\tau \wedge \varepsilon^2} \left| V_j(t+\tau,x',\lambda;\varepsilon) - V_j(t,x',\lambda;\varepsilon)\right| \, d \lambda
\]
\[
+ \sum_{j=1}^3 \int_{\tau \wedge \varepsilon^2}^{\varepsilon^2} \left| V_j(t+\tau,x',\lambda;\varepsilon) - V_j(t,x',\lambda;\varepsilon)\right| \, d \lambda
\]
\[
=: \Delta_\tau U_1 + \Delta_\tau U_2 + \Delta_\tau U_3.
\]
This gives
\begin{equation}
							\label{eq0712_01}
\| u(\cdot + \tau, 0, \cdot) - u(\cdot,0,\cdot)\|_p \le \| \Delta_\tau U_1\|_p + \| \Delta_\tau U_2\|_{L_p} + \| \Delta_\tau U_3\|_p,
\end{equation}
where $\|\Delta_\tau U_3\|_p = 0$ if $\tau \geq \varepsilon^2$.
Thus, to derive the desired estimate in the lemma, we deal with the terms on the right-hand side of the inequality \eqref{eq0712_01}.

First, for $\Delta_\tau U_1$, we write
$$
\Delta_\tau U_1 = \left| \int_0^{\tau} \frac{\partial u^{(\varepsilon)}}{\partial t}(t + \ell, 0, x') \,d\ell     \right|,
$$
where, by the fact that $u_t = D_i g_i$ in $\bR \times \bR^n_+$ and $\phi_\varepsilon(t,x) = \eta_\varepsilon(x_1) \zeta_\varepsilon(t,x')$ with
\[
\eta_\varepsilon(x_1) = \varepsilon^{-1}\eta(x_1/\varepsilon), \quad
\zeta_\varepsilon(t,x') = \varepsilon^{-n-1}\zeta(t/\varepsilon^2,x'/\varepsilon),
\]
we have
\begin{align*}
\frac{\partial u^{(\varepsilon)}}{\partial t}(t,0,x') &= \int_{\bR \times \bR^n_+} u(s,y) \frac{\partial}{\partial t}\left[ \phi_\varepsilon(t-s,-y_1,x'-y')\right] \, dy \, ds
\\
&= \int_{\bR \times \bR^n_+} g_1(s,y) \eta_\varepsilon'(-y_1) \zeta_\varepsilon(t-s,x'-y') \, dy \, ds
\\
&+ \sum_{i=2}^n \int_{\bR \times \bR^n_+} g_i(s,y) \eta_\varepsilon(-y_1) D_i\zeta_\varepsilon(t-s,x'-y') \, dy \, ds.
\end{align*}
Then by Minkowski's inequality, Young's convolution inequality, H\"older's inequality, and the fact that $\operatorname{supp} \eta \subset (-1,-1/2)$,
\begin{align*}
\|\Delta_\tau U_1 \|_p &\le N \tau \varepsilon^{-2} \int_0^{\infty} \|g_i(\cdot, y_1, \cdot)\|_p  \left(  \left| \eta \right| + \left| \eta' \right|  \right) \left( -\frac{y_1}{\varepsilon} \right) dy_1 
\\
&\le N \tau \varepsilon^{-2} \int_{\varepsilon/2}^{\varepsilon}\|g_i(\cdot, y_1, \cdot)\|_p \, y_1^{\kappa}y_1^{-\kappa} \,dy_1
\\
&\le N \tau \varepsilon^{-2} \left( \int_{\varepsilon/2}^{\varepsilon} y_1^{-q\kappa} \, dy_1\right)^{1/q} \left( \int_{\varepsilon/2}^{\varepsilon}  \|g_i(\cdot, y_1, \cdot)\|_p^p \, y_1^{p\kappa}      \, dy_1\right)^{1/p} 
\\
&\le N \tau \varepsilon^{-1 -\kappa - 1/p} \left( \int_{\varepsilon/2}^{\varepsilon}  \|g_i(\cdot, y_1, \cdot)\|_p^p \, y_1^{p\kappa}      \, dy_1\right)^{1/p},
\end{align*}
where $1/p + 1/q = 1$ and $\kappa = 1 - \mathsf{s} - 1/p$.
Hence,
\begin{equation}
							\label{eq0401_01}
\begin{aligned}
&\int_0^T \tau^{-1-\mathsf{s}p/2} \|\Delta_\tau U_1\|_{L_p }^p \, d \tau
\\
&\le N \varepsilon^{-p - p\kappa -1} \int_0^T \tau^{-1-\mathsf{s}p/2 + p} \int_{0}^{\varepsilon}  \|g_i(\cdot, y_1, \cdot)\|_p^p \, y_1^{p\kappa}      \, dy_1 \,d\tau
\\
&\le N T^{p-\mathsf{s}p/2}\varepsilon^{-p - p\kappa -1} \int_{0}^{\varepsilon}  \|g_i(\cdot, y_1, \cdot)\|_p^p \, y_1^{p\kappa}\,dy_1
\\
&= N \left( \frac{T}{\varepsilon^2}\right)^{p -\mathsf{s}p/2} \int_{0}^{\varepsilon}  \|g_i(\cdot, y_1, \cdot)\|_p^p \, y_1^{\theta - n}\,dy_1
\end{aligned}
\end{equation}
where $N = N(n, p, \mathsf{s}) = N(n, p, \theta)$.

We now estimate $\Delta_\tau U_2$.
Take $\alpha \in \bR$ such that
\[
\frac{5}{q} - \mathsf{s} < \alpha < \frac{5}{q}, \quad \text{where} \quad 1/p+1/q = 1.
\]
From the definition of $V_1$ in \eqref{eq0329_01}, Minkowski's inequality, Young's convolution inequality, and H\"older's inequality (also recall $\phi(t,x) = \eta(x_1)\zeta(t,x')$ and $\operatorname{supp}\eta \subset (-1,-1/2)$), it follows that
\[
\left\| \int_0^\tau |V_1(\cdot,\cdot,\lambda;\varepsilon)|\, d \lambda\right\|_p
\]
\[
\leq N \int_0^\tau \lambda^{-1} \int_{\frac{\sqrt{\lambda}}{2}}^{\sqrt{\lambda}} \frac{y_1}{\sqrt{\lambda}} \eta\left(-\frac{y_1}{\sqrt{\lambda}} \right) \|Du(\cdot,y_1,\cdot)\|_p \, y_1^\alpha y_1^{-\alpha} \, dy_1 \, d\lambda
\]
\[
\leq N \left(\int_0^\tau \int_{\frac{\sqrt{\lambda}}{2}}^{\sqrt{\lambda}} \lambda y_1^{-q \alpha} \, dy_1 \, d\lambda \right)^{1/q} \left(\int_0^\tau \int_{\frac{\sqrt{\lambda}}{2}}^{\sqrt{\lambda}} \lambda^{1-2p} \|Du(\cdot,y_1,\cdot)\|_p^p \, y_1^{p \alpha} \, dy_1 \, d\lambda \right)^{1/p}
\]
\[
= N \tau^{\frac{5}{2q}-\frac{\alpha}{2}}\left(\int_0^\tau \int_{\frac{\sqrt{\lambda}}{2}}^{\sqrt{\lambda}} \lambda^{1-2p} \|Du(\cdot,y_1,\cdot)\|_p^p \, y_1^{p \alpha} \, dy_1 \, d\lambda \right)^{1/p},
\]
where we used $1/q + (1-2p)/p = -1$, $5 - q\alpha > 0$, and the fact that
\[
\frac{y_1}{\sqrt{\lambda}} \eta\left(-\frac{y_1}{\sqrt{\lambda}} \right)
\]
is bounded by a constant depending on the choice of $\eta$.
Similar calculations for $V_2$ and $V_3$ with the above estimate for $V_1$ show that
\[
\|\Delta_\tau U_2\|_p \leq \sum_{j=1}^3\left\|\int_0^\tau |V_j(\cdot+\tau,\cdot,\lambda;\varepsilon) - V_j(\cdot,\cdot,\lambda;\varepsilon)| \, d\lambda \right\|_p
\]
\[
\leq 2 \sum_{j=1}^3\left\| \int_0^\tau |V_j(\cdot,\cdot,\lambda;\varepsilon)| \, d\lambda\right\|_p
\]
\[
\leq N \tau^{\frac{5}{2q}-\frac{\alpha}{2}}\left(\int_0^\tau \int_{\frac{\sqrt{\lambda}}{2}}^{\sqrt{\lambda}} \lambda^{1-2p} \|G(\cdot,y_1,\cdot)\|_p^p \, y_1^{p \alpha} \, dy_1 \, d\lambda \right)^{1/p},
\]
where $G(t, x) = \left| Du(t, x)\right| + \left| g(t, x)\right|$ and, by changing the order of the integrations, the last double integral is not greater than
\[
\int_0^{\sqrt{\tau}}  \int_{y_1^2}^{4y_1^2} \lambda^{1-2p} \, d\lambda \|G(\cdot,y_1,\cdot)\|_p^p \, y_1^{p \alpha} \, dy_1
\]
\[
\leq N \int_0^{\sqrt{\tau}} \|G(\cdot,y_1,\cdot)\|_p^p \, y_1^{p \alpha - 4p+4} \, dy_1.
\]
Therefore,
\begin{equation}
							\label{eq0401_02}
\begin{aligned}
&\int_0^T \tau^{-1-\mathsf{s}p/2} \|\Delta_\tau U_2\|_{L_p}^p \,d\tau
\\
&\leq N\int_0^T \tau^{\frac{5p}{2} - \frac{p\alpha}{2}-\frac{\mathsf{s}p}{2} - \frac{7}{2}} \int_0^{\sqrt{\tau}} \|G(\cdot,y_1,\cdot)\|_p^p \, y_1^{p \alpha - 4p+4} \, dy_1 \, d\tau
\\
&\leq N \int_0^{\sqrt{T}} \|G(\cdot,y_1,\cdot)\|_p^p \, y_1^{p \alpha - 4p+4}  \int_{y_1^2}^\infty \tau^{\frac{5p}{2} - \frac{p\alpha}{2}-\frac{\mathsf{s}p}{2} - \frac{7}{2}} \, d \tau \, dy_1
\\
&= N \int_0^{\sqrt{T}} \|G\left( \cdot, y_1, \cdot \right)\|_p^p \, y_1^{\theta - n} \, dy_1,
\end{aligned}
\end{equation}
where $N = N(n,p,\theta)$, and the last equality is due to $5/q - \mathsf{s} < \alpha$ and $p-\mathsf{s}p-1 = \theta - n$.

Finally, to estimate $\Delta_\tau U_3$, which is only needed when $\tau < \varepsilon^2$, 
we write
\[
V_j(t+\tau,x',\lambda;\varepsilon) - V_j(t,x',\lambda;\varepsilon) = \int_0^\tau \frac{\partial V_j}{\partial t}(t+\ell,x',\lambda;\varepsilon) \, d\ell.
\]
Note that
\[
\frac{\partial V_1}{\partial t}(t,x',\lambda;\varepsilon) = \frac{\lambda^{-\frac{n}{2}-\frac{5}{2}}}{2} \int_0^\infty \frac{z_1}{\sqrt{\lambda}} \eta\left(-\frac{z_1}{\sqrt{\lambda}}\right) D_1u(\cdot,z_1,\cdot) * \zeta_t\left(\frac{\cdot}{\lambda},\frac{\cdot}{\sqrt{\lambda}}\right) \, dz_1,
\]
where, here and below,
$v*w$ indicates the convolution of $v$ and $w$ as functions on $\bR \times \bR^{n-1} = \{(t,x'): t \in \bR, x' \in \bR^{n-1}\}$,
\[
\frac{\partial V_2}{\partial t} = \frac{\lambda^{-\frac{n}{2}-\frac{5}{2}}}{2} \sum_{i=2}^n \int_0^\infty \eta\left(-\frac{z_1}{\sqrt{\lambda}}\right) D_iu(\cdot,z_1,\cdot) * \tilde{\zeta}_i \left(\frac{\cdot}{\lambda},\frac{\cdot}{\sqrt{\lambda}}\right) \, dz_1,
\]
where $\tilde{\zeta}_i(t,x) = - x_i \zeta_t(t,x')$, and
\begin{align*}
\frac{\partial V_3}{\partial t} & = \lambda^{-\frac{n}{2}-\frac{5}{2}} \int_0^\infty \eta'\left(-\frac{z_1}{\sqrt{\lambda}}\right) g_1(\cdot,z_1,\cdot) * \hat{\zeta}\left(\frac{\cdot}{\lambda},\frac{\cdot}{\sqrt{\lambda}}\right) \, dz_1
\\
&+ \lambda^{-\frac{n}{2}-\frac{5}{2}} \sum_{i=2}^n \int_0^\infty \eta\left(-\frac{z_1}{\sqrt{\lambda}}\right) g_i(\cdot,z_1,\cdot) * \check{\zeta}_i\left(\frac{\cdot}{\lambda},\frac{\cdot}{\sqrt{\lambda}}\right) \, dz_1,
\end{align*}
where $\hat{\zeta}(t,x') = - \zeta(t,x') -t\zeta_t(t,x')$ and $\check{\zeta}_i(t,x') = - D_i\zeta(t,x') - t D_i\zeta_t(t,x')$.
Then by Minkowski's inequality
\[
\|\Delta_\tau U_3\|_p \leq \sum_{j=1}^3 \int_\tau^{\varepsilon^2} \int_0^\tau \left\| \frac{\partial V_j}{\partial t}(\cdot+\ell,\cdot,\lambda;\varepsilon)  \right\|_p \, d\ell \, d\lambda =: \sum_{j=1}^3 I_j,
\]
where, for instance, by Young's convolution inequality and H\"older's inequality
\[
\left\| \frac{\partial V_1}{\partial t}(\cdot+\ell,\cdot,\lambda;\varepsilon)  \right\|_p \leq N \lambda^{-2} \int_0^\infty \frac{z_1}{\sqrt{\lambda}} \eta\left(-\frac{z_1}{\sqrt{\lambda}}\right) \|D_1u(\cdot,z_1,\cdot)\|_p \, dz_1
\]
\[
\leq N \lambda^{-2} \int_{\frac{\sqrt{\lambda}}{2}}^{\sqrt{\lambda}} \|D_1u(\cdot,z_1,\cdot)\|_p \, dz_1.
\]
Thus,
\[
I_1 \leq N \tau \int_\tau^{\varepsilon^2} \int_{\frac{\sqrt{\lambda}}{2}}^{\sqrt{\lambda}}  \lambda^{-2} \|D_1u(\cdot,z_1,\cdot)\|_p \, dz_1 \, d\lambda
\]
\[
= N \tau \int_\tau^{\varepsilon^2} \int_{\frac{\sqrt{\lambda}}{2}}^{\sqrt{\lambda}}  \lambda^{-\frac{1}{2}-\frac{1}{q}} \, z_1^{-\frac{1}{q}} \, \lambda^{-\frac{1}{2}-\frac{1}{p}} \|D_1u(\cdot,z_1,\cdot)\|_p \, z_1^{\frac{1}{q}} \, dz_1 \, d\lambda
\]
\[
\leq N \tau \left( \iint \ldots \right)^{\frac{1}{q}} \left(\int_\tau^{\varepsilon^2} \int_{\frac{\sqrt{\lambda}}{2}}^{\sqrt{\lambda}} \lambda^{-\frac{p}{2}-1} \|D_1u(\cdot,z_1,\cdot)\|_p^p \, z_1^{\frac{p}{q}} \, dz_1 \, d\lambda\right)^{\frac{1}{p}},
\]
where
\[
\iint \ldots = \int_\tau^{\varepsilon^2} \int_{\frac{\sqrt{\lambda}}{2}}^{\sqrt{\lambda}} \lambda^{-\frac{q}{2}-1} z_1^{-1} \, dz_1 \, d\lambda \leq N \tau^{-\frac{q}{2}}.
\]
The above estimate for $I_1$ along with similar calculations for $I_2$ and $I_3$ gives
\begin{equation}
							\label{eq0401_03}
\begin{aligned}
&\int_0^T \tau^{-1-\frac{\mathsf{s}p}{2}} \|\Delta_\tau U_3\|_p^p \, d\tau
\\
&\leq N \int_0^{\varepsilon^3} \tau^{\frac{p}{2}-1-\frac{\mathsf{s}p}{2}}\int_\tau^{\varepsilon^2} \int_{\frac{\sqrt{\lambda}}{2}}^{\sqrt{\lambda}} \lambda^{-\frac{p}{2}-1} \|G(\cdot,z_1,\cdot)\|_p^p \, z_1^{\frac{p}{q}} \, dz_1 \, d\lambda \, d\tau
\\
&= N \int_0^{\varepsilon^2} \int_0^\lambda \tau^{\frac{p}{2}-1-\frac{\mathsf{s}p}{2}} \, d\tau \int_{\frac{\sqrt{\lambda}}{2}}^{\sqrt{\lambda}} \lambda^{-\frac{p}{2}-1} \|G(\cdot,z_1,\cdot)\|_p^p \, z_1^{\frac{p}{q}} \, dz_1\, d\lambda
\\
&= N \int_0^{\varepsilon^2} \int_{\frac{\sqrt{\lambda}}{2}}^{\sqrt{\lambda}} \lambda^{-\frac{\mathsf{s}p}{2}-1} \|G(\cdot,z_1,\cdot)\|_p^p \, z_1^{\frac{p}{q}} \, dz_1\, d\lambda
\\
&\leq N \int_0^\varepsilon \|G(\cdot,z_1,\cdot)\|_p^p \, z_1^{\theta - n} \, dz_1,
\end{aligned}
\end{equation}
where, as above, $G = |Du| + |g|$ and we used the fact that $s \in (0,1)$.

From the estimates \eqref{eq0401_01}, \eqref{eq0401_02}, \eqref{eq0401_03} with \eqref{eq0712_01}, we see that the inequality \eqref{eq0405_01} holds.
The lemma is proved.
\end{proof}

\begin{proof}[Proof of Proposition \ref{prop0712_01}]
By Lemma \ref{lem0901}, it is enough to obtain \eqref{eq1001} for $u$ as in Lemma \ref{lem0712_01}. That is, we assume that $u \in C_0^\infty(\overline{\bR \times \bR^n_+})$ with $u_t = D_ig_i$ in $\bR \times \bR^n_+$, $g_i \in \bL_{p, \theta}(\bR)$.  
From Lemma \ref{lem0603_2}, we have
$$
\|u(\cdot, 0,\cdot)\|_{L_p(\bR \times \bR^{n-1})} + \|u(\cdot, 0,\cdot)\|_{W_p^{0,\mathsf{s}}(\bR \times \bR^{n-1})}
$$
$$
\le N \left(\|u\|_{\bL_{p, \theta}(\bR)} + \|Du\|_{\bL_{p, \theta}(\bR)}\right),
$$
where $N=N(n, p, \theta)$. 
By taking $\varepsilon = T^{1/2}$ in Lemma \ref{lem0712_01} and letting $T \to \infty$, we obtain (also recall \eqref{eq0401_04})
$$
\|u(\cdot, 0,\cdot)\|_{W_p^{\mathsf{s}/2,0}(\bR \times \bR^{n-1})} \le N \left(\|Du\|_{\bL_{p, \theta}(\bR)} + \|g_i\|_{\bL_{p, \theta}(\bR)}\right),
$$
where $N=N(n, p, \theta)$.
By combining above two inequalities we arrive at \eqref{eq1001}.
The proposition is proved.
\end{proof}

\begin{remark}
							\label{rem0426_1}

In Proposition \ref{prop0712_01}, we have $u_t \in \bH_{p, \theta + p}^{-1}(\bR)$, which means that $u_t$ may not be a function.
Nevertheless, via an integral representation of $u$, we make use of the fact that $u_t = D_ig_i$ in $\bR \times \Omega$ in `distribution' sense for some $g_i \in \bL_{p, \theta}(\bR)$, $i = 1. \ldots, n$.
On the other hand, if $u \in \widetilde{\fH}_{p, \theta}^{2}(\bR)$ so that $u_t \in \bL_{p,\theta+p}(\bR)$, one can give a considerably simple proof of Proposition \ref{prop0712_01} (in fact, Lemma \ref{lem0712_01}), without using any integral representations,  by proceeding as in the proof of Lemma \ref{lem0603_2} with $u_t$ in place of $D_1 u$.
\end{remark}

\section{Extension operator}
	\label{exten}
In this section, except Corollary \ref{extension_lip}, we set $\Omega = \bR^n_+$.
Recall that, in this case, for instance, by $\bL_{p,\theta}(\bR)$ we mean
\[
\|u\|_{\bL_{p,\theta}(\bR)}^p = \int_{\bR} \|u(t,\cdot)\|_{L_{p,\theta}(\bR^n_+)}^p \, dt = \int_\bR \int_{\bR^n_+} |u(t,x)|^p x_1^{\theta -n} \, dx \, dt.
\]

\begin{proposition}
							\label{prop0715_01}
Let $1 < p < \infty$, $n-1 < \theta < n-1+p$, and $\mathsf{s} =\left( n-1+p-\theta \right)/p$.
Then for $g \in W_p^{\mathsf{s}/2,\mathsf{s}}(\bR \times \bR^{n-1})$ and $\gamma \in \bR$, there exists $u \in \widetilde{\fH}_{p,\theta}^\gamma (\bR)$ such that
$$
\cT u = g,
$$
where $\cT$ is the trace operator from $\widetilde{\fH}_{p,\theta}^1 (\bR)$ to $W_p^{\mathsf{s}/2, \mathsf{s}}(\bR \times \bR^{n-1})$ in Proposition \ref{prop0712_01}.
Moreover,
\begin{equation}
	\label{eq0718_01}
\|u\|_{\widetilde{\fH}_{p,\theta}^\gamma (\bR)} \le N\|g\|_{W_p^{\mathsf{s}/2,\mathsf{s}}(\bR \times \bR^{n-1})},
\end{equation}
where $N=N(n, p, \theta, \gamma)$.
\end{proposition}

Before we present the proof of Proposition \ref{prop0715_01}, let us recall that, for a sufficiently smooth $g(t,x')$, a solution $u$ to the problem $u_t = \Delta u$ in $\bR \times \bR^n_+$ with the lateral boundary condition $u(t,0,x') = g(t,x')$ on $\bR \times \bR^{n-1}$ is given by
\begin{equation}
							\label{eq0406_01}
\begin{aligned}
u(t,x_1,x') &=[g*p(\cdot,x_1,\cdot)](t,x')
\\
&=\int_{-\infty}^{\infty}\int_{\bR^{n-1}}p(t-s,x_1,x'-y')g(s,y') \, dy' \, ds,
\end{aligned}
\end{equation}
where
$$
p(t,x_1,x')=1_{t>0}
\frac{1}{(4\pi t)^{n/2}}\frac{x_1}{t}e^{-\frac{|x|^2}{4t}}
$$
for $(t,x) \in \bR \times \bR^n_+$. Note that $p$ is infinitely differentiable in $\bR \times \bR^n_+$,
\begin{equation}
          \label{eqn 160112.8}
\int_{\bR} \int_{\bR^{n-1}}p(t,x_1,x') \, dx' \, dt=1, \quad  \quad p_t=\Delta p,
\end{equation}
and, for any multi-index $\alpha=(\alpha_1,\cdots,\alpha_n)$ with respect to the spatial variables, $|\alpha|\geq 1$,
\begin{equation}
          \label{eqn 114.6}
\int_{\bR} \int_{\bR^{n-1}}D^{\alpha}p(t,x_1,x') \, dx' \, dt=0.
\end{equation}
Indeed, since $D^{\alpha}p(t,x_1,x')$ is integrable on $\{(t,x') \in \bR \times \bR^{n-1}\}$ for any $\alpha$ (see (\ref{e1}) and (\ref{e2})), one can check easily that \eqref{eqn 114.6} holds when $\alpha = (0, \alpha_2,\ldots, \alpha_n)$ or the integrand is $p_t$.
For $p_{x_1}$, a direct calculation of $p_{x_1}$ gives \eqref{eqn 114.6}.
Then, the remaining cases follow from the relation $p_t = \Delta p$ and the aforementioned cases.
Also note that $u$ is infinitely differentiable in $\bR \times \bR^n_+$ and $u(t,x)=0$ for $t\leq 0$ if $g(t,x')=0$ for $t\leq 0$.

\begin{lemma}
							\label{lem0603_03}
Let $p$, $\theta$ and $\mathsf{s}$ be as in Proposition \ref{prop0715_01}.
If $g \in C_0^\infty(\bR \times \bR^{n-1})$ and $u$ is defined as in \eqref{eq0406_01}, then for any $m=0,1,2,\ldots$ and  multi-index $\alpha$ with respect to the spatial variables (i.e., $D^\alpha = D_x^{\alpha}$), we have
\begin{equation}
							\label{eq0409_01}
\| D^{\alpha}u\|_{\bL_{p,\theta + mp}(\bR)}\leq N \left( [ g ]_{W^{\mathsf{s}/2, 0}(\bR \times \bR^{n-1})} +  [ g ]_{W^{0, \mathsf{s}}(\bR \times \bR^{n-1})}\right)
\end{equation}
for $|\alpha| = m+1$, and
$$
\|D^{\alpha}u_t\|_{\bL_{p,\theta + (m+1)p }(\bR )}\leq N \left( [g]_{W^{\mathsf{s}/2, 0}(\bR \times \bR^{n-1})} + [g]_{W^{0, \mathsf{s}}(\bR \times \bR^{n-1})}\right)
$$
for $|\alpha| = m$,
where $N=N(n, p, \theta, m)$.
\end{lemma}

\begin{proof}
Since the second inequality is a direct consequence of the first one with the fact $u_t = \Delta u$ in $\bR \times \bR^n_+$, we only prove the first inequality. 

From \eqref{eq0406_01} and \eqref{eqn 114.6}, we have
$$
D^{\alpha}u(t,x)=\int^{\infty}_0 \int_{\bR^{n-1}} D^{\alpha}p(s,x_1,y')
\left(g(t-s,x'-y')-g(t,x')\right) \, dy' \, ds,
$$
where by the the triangle inequality
\[
|g(t-s,x'-y')-g(t,x')|
\]
\[
\leq |g(t-s,x'-y')-g(t,x'-y')| + |g(t,x'-y') - g(t,x')|.
\]
Thus, by Minkowski's inequality it follows that
\begin{multline}
							\label{eq0409_02}
\|D^\alpha u(\cdot,x_1,\cdot)\|_p \leq  \int_0^\infty \int_{\bR^{n-1}} |D^\alpha p(s,x_1,y')| K_1(s) \, dy'\, ds
\\
+ \int_0^\infty \int_{\bR^{n-1}} |D^\alpha p(s,x_1,y')| K_2(y') \, dy'\, ds,
\end{multline}
where $\|\cdot\|_p = \|\cdot\|_{L_p(\bR \times \bR^{n-1})}$ and
\[
K_1(s) = \| g(\cdot-s,\cdot) - g(\cdot,\cdot)\|_p, \quad K_2(y') = \|g(\cdot,\cdot-y') - g(\cdot,\cdot)\|_p.
\]

Using induction, we see that $D^\alpha p(t,x_1,x')$ with $|\alpha|=m+1$ is a linear combination of the following terms with some non-negative integers $a$ and $b$:
\begin{align}
&I_1(m+1,a,b,\beta)  \nonumber
\\
&=x_1^{1+b} t^{-(n/2+m+2-a)} e^{-(x_1^2+|x'|^2)/4t} (x')^{\beta} \times 1_{|\beta|=m+1-2a-b}\,1_{t>0} \label{e1},
\end{align}
\begin{align}
&I_2(m+1,a,b,\beta) \nonumber
\\
&=x_1^{b} \, t^{-(n/2+m+1-a)}e^{-( {x_1^2+|x'|^2})/4t} (x')^{\beta}  \times 1_{|\beta|=m-2a-b}\,1_{t>0}, \label{e2}
\end{align}
where $\beta = (\beta_2,\ldots,\beta_n)$ is a multi-index with respect to $x'$.
For instance, $D_{x_1} p(t,x_1,x')$ is a linear combination of the term \eqref{e1} with $a=0$ and $b=1$, and the term \eqref{e2} with $a=0$ and $b=0$.
Hence, by \eqref{e1}, \eqref{e2}, and \eqref{eq0409_02}, to prove \eqref{eq0409_01} it is enough to show that, for $i,j=1,2$,
\begin{multline}
							\label{eq0409_03}
\int_0^\infty x_1^{\theta + mp - n} \left( \int_0^\infty \int_{\bR^{n-1}} I_j(m+1,a,b,\beta) K_i(s,y') \, dy' \, ds \right)^p dx_1
\\
\leq N \left( [g]_{W^{\mathsf{s}/2, 0}}^p + [g]_{W^{0, \mathsf{s}}}^p\right),
\end{multline}
where $K_i(s,y')$ is either $K_1(s)$ or $K_2(y')$.
Furthermore, thanks to the relation $I_1(m+1,a,b,\beta) = x_1 I_2(m+2,a,b,\beta)$, we only prove \eqref{eq0409_03} when $j = 2$.
We write
\[
I_2(m+1,a,b,\beta) = x_1^b J(m+1,a,b,\beta)
\]
and, using the condition that 
\begin{equation}
							\label{eq0409_05}
 n-1 < \theta < n-1+p,
\end{equation}
find a positive number $\delta_1$ such that
\begin{equation}
							\label{eq0409_04}
0 < \delta_1 < \frac{\theta - n + 1}{(p-1)(n+m+b+2)}.
\end{equation}
Then, set
\[
B = \frac{n+1}{n+m+b+2} + \delta_1,
\]
and find a number $A$ such that
\[
\frac{A}{p} + \frac{B}{q} = 1, \quad \text{where} \quad \frac{1}{p} + \frac{1}{q} = 1.
\]
By the choice of $\delta_1$ and \eqref{eq0409_05}, $A > 0$.
To prove \eqref{eq0409_03} with $j=2$, by H\"{o}lder's inequality, we have
\begin{equation}
							\label{eq0410_01}
\begin{aligned}
&\left( \int_0^\infty \int_{\bR^{n-1}} I_2(m+1,a,b,\beta) K_i(s,y') \, dy' \, ds \right)^p
\\
&\leq x_1^{bp} \left(\int_0^\infty \int_{\bR^{n-1}}|K_i|^p J^A \, dy' \, ds \right) \left(\int_0^\infty \int_{\bR^{n-1}} J^B  \, dy' \, ds\right)^{p/q}
\\
&:= x_1^{bp} \Theta_1 \Theta_2^{p/q}.
\end{aligned}
\end{equation}
Here we have that, since $B > (n+1)/(n+m+b+2)$,
\[
\Theta_2 \leq N x_1^{-B(n+m+b+2) + n + 1} \int_0^\infty s^{-\frac{B}{2}(n + m + b + 2) + \frac{n}{2} - \frac{1}{2} } e^{-\frac{B}{4s}} \, ds
\]
\[
= N x_1^{-B(n+m+b+2) + n + 1}.
\]
Regarding $\Theta_1$ for $K_i = K_1(s)$, we have
\[
\Theta_1 \leq \int_0^\infty |K_1(s)|^p s^{-A\left(\frac{n}{2}+m+1-a\right)} e^{-\frac{Ax_1^2}{4s}} \int_{\bR^{n-1}} e^{-\frac{A|y'|^2}{4s}} |y'|^{A|\beta|} \, dy' \, ds
\]
\[
= N \int_0^\infty |K_1(s)|^p s^{-\frac{A}{2}(n+m+b+2) + \frac{n}{2} - \frac{1}{2}} e^{-\frac{Ax_1^2}{4s}} \, ds.
\]
Thus, the left-hand side of \eqref{eq0409_03} with $I_2$ and $K_1$ is not greater than a constant times
\[
\int_0^\infty x_1^{\kappa_1} \int_0^\infty |K_1(s)|^p s^{-\frac{A}{2}(n+m+b+2) + \frac{n}{2} - \frac{1}{2}} e^{-\frac{Ax_1^2}{4s}} \, ds \, dx_1
\]
\[
= \int_0^\infty s^{-\frac{A}{2}(n+m+b+2) + \frac{n}{2} - \frac{1}{2}} |K_1(s)|^p \int_0^\infty x_1^{\kappa_1} e^{-\frac{Ax_1^2}{4s}} \, dx_1 \, ds
\]
\[
= N\int_0^\infty s^{-1 - \frac{1}{2}(n - 1 + p - \theta)} |K_1(s)|^p \, ds \leq N [g]^p_{W_p^{\mathsf{s}/2,0}(\bR \times \bR^{n-1})},
\]
where, by the choice of $B$ with \eqref{eq0409_05},
\[
\kappa_1 = \theta + mp - n + bp -\frac{p}{q}B(n+m+b+2) + \frac{p}{q}(n+1) > - 1.
\]
The above inequality proves \eqref{eq0409_03} for $j=2$ and $i=1$.
To show \eqref{eq0409_03} with $j=i=2$, since $K_2(y') = 0$ when $n=1$, we assume $n \geq 2$.
Then, upon recalling the inequalities in \eqref{eq0410_01} and the above estimate for $\Theta_2$, we see that the left-hand side of \eqref{eq0409_03} with $K_2(y')$ is not greater than a constant times
\[
\int_0^\infty x_1^{\kappa_1} \int_0^\infty \int_{\bR^{n-1}} |K_2(y')|^p s^{-A\left(\frac{n}{2}+m+1-a\right)} e^{-\frac{A(x_1^2+|y'|^2)}{4s}} |y'|^{A|\beta|} \, dy' \, ds \, dx_1
\]
\[
= N \int_{\bR^{n-1}} |K_2(y')|^p|y'|^{A(m-2a-b)} \int_0^\infty s^{\kappa_2} e^{-\frac{A|y'|^2}{4s}} \, ds \, dy'
\]
\[
= N \int_{\bR^{n-1}} |K_2(y')|^p|y'|^{A(m-2a-b) + 2 \kappa_2 + 2} \, dy' = N [g]_{W_p^{0,\mathsf{s}}(\bR \times \bR^{n-1})}^p,
\]
where, by the choice of $A$, the definition of $\kappa_1$, \eqref{eq0409_05}, and the fact that $n\geq 2$, 
\[
\kappa_2: = -A\left(\frac{n}{2}+m+1-a\right) + \frac{\kappa_1}{2} + \frac{1}{2} = \frac{\theta}{2} - \frac{p}{2} - n - \frac{A}{2}|\beta| < - 1,
\]
\[
A(m-2a-b)+2\kappa_2+2 = \theta - p - 2n + 2 = - (n - 1 + \mathsf{s}p).
\]
Hence, the inequality \eqref{eq0409_03} is also proved for $j=i=2$.
Therefore, the inequality \eqref{eq0409_01} is proved, and so is the lemma.
\end{proof}

\begin{proof}[Proof of Proposition \ref{prop0715_01}]
For $\gamma$, take an integer $k$ such that $k \ge \max\{\gamma, 2\}$.
Since $C_0^{\infty}(\bR \times \bR^{n-1})$ is dense in $W_p^{s/2, s}(\bR \times \bR^{n-1})$, it is enough to find $u \in \widetilde{\fH}_{p, \theta}^{k}(\bR)$ satisfying \eqref{eq0718_01} for $g \in C_0^{\infty}(\bR \times \bR^{n-1}) \cap W_p^{\mathsf{s}/2, \mathsf{s}}(\bR \times \bR^{n-1})$.

Set $v = g  *  p(\cdot,x_1,\cdot)(t,x')$ as in Lemma \ref{lem0603_03} and let $u(t,x) := v(t, x) \zeta(x_1)$ where $\zeta \in C_0^{\infty}(\overline{\bR_+})$ such that $\zeta(x_1) = 1$ for $x_1 \le 1$ and $\zeta(x_1) = 0$ for $x_1 \ge 2$.
Note that from \eqref{eqn 160112.8} and Young's inequality for convolutions,
\[
\|u\|_{\bL_{p,\theta}(\bR)}^p \leq \int_0^2 \|v(\cdot,x_1,\cdot)\|_p^p \, x_1^{\theta - n} \, dx_1 \leq N \|g\|_p,
\]
where $N = N(n, p, \theta)$, $\|\cdot\|_p = \|\cdot\|_{L_p(\bR \times \bR^{n-1})}$, and we used the fact that $\theta - n > - 1$.
On the other hand, by Lemma \ref{lem0603_03}, we have
$$
\|Dv\|_{\bH_{p,\theta}^{k-1}(\bR)} + \|v_t\|_{\bH_{p,\theta}^{k-2}(\bR)}\le N  \left( [g]_{W_p^{\mathsf{s}/2,0}(\bR \times \bR^{n-1})} + [g]_{W_p^{0,\mathsf{s}}(\bR \times \bR^{n-1})} \right),
$$
where $N=N(n, p, \theta, k) = N(n, p, \theta, \gamma)$.
Then the desired inequality follows by combining the above two inequalities, the fact that $x_1^{\theta - n + mp} \leq N(m) x_1^{\theta - n + jp}$ for $x_1 \in (0,2)$ and $j \in \{0, \ldots,m\}$, and the inequality (see Proposition \ref{propertyweighted})
$$
\|u\|_{\widetilde{\fH}_{p,\theta}^\gamma(\bR)} \le \|u\|_{\widetilde{\fH}_{p,\theta}^k(\bR)}.
$$
\end{proof}

\begin{remark}
	\label{rmk0117}
Note that $u(=v\zeta)$ in the proof of Proposition \ref{prop0715_01} satisfies 
$$
\| u \|_{\mathcal{W}_{p, \theta + kp}^{k+1}(\bR )}  \le N\| u \|_{\widetilde{\fH}_{p, \theta}^{k+1}(\bR)}  \le N(n, p, \theta, k) \| g \|_{W_p^{\mathsf{s}/2, \mathsf{s}}(\bR \times \bR^{n-1})}
$$
for any positive integer $k$, since $u$ satisfies \eqref{eq0718_01} for any $\gamma \in \bR$ and $u(t, x) = 0$ when $x_1 \ge 2$.
\end{remark}

Once we have Proposition \ref{prop0715_01}, by applying a flattening argument with the help of \cite[Lemma 2.14]{MR2352844} and Proposition \ref{propertyweighted} (4), we obtain the following extension result for $\partial\Omega \in C^{0, 1}$.

\begin{corollary}
	\label{extension_lip}
Let $p$, $\theta$ and $\mathsf{s}$ be as in Proposition \ref{prop0715_01} and $\partial\Omega \in C^{0, 1}$.
Then for $g \in W_p^{\mathsf{s}/2,\mathsf{s}}(\bR \times \partial\Omega)$ and $\gamma \in \bR$, there exists
$u \in \widetilde{\fH}_{p,\theta}^\gamma (\bR)$ such that
$$
\cT u = g,
$$
where $\cT$ is the trace operator from $\widetilde{\fH}_{p,\theta}^1 (\bR)$ to $W_p^{\mathsf{s}/2, \mathsf{s}}(\bR \times \partial\Omega )$ as in Corollary \ref{trace_optimal}.
Moreover,
\begin{equation}
\|u\|_{\widetilde{\fH}_{p,\theta}^\gamma (\bR)} \le N\|g\|_{W_p^{\mathsf{s}/2,\mathsf{s}}(\bR \times \partial\Omega )},
\end{equation}
where $N=N(n, p, \theta, \gamma, M, K, R)$.
\end{corollary}

\section{Proof of Theorem \ref{thm0716_01}}
							\label{main}

In Section \ref{trace} we proved the trace part of Theorem \ref{thm0716_01} when $\gamma = 1$. In this section we prove the trace theorem for any $\gamma \geq 1$, which together with the extension theorems in Section \ref{exten} completes the proof of Theorem \ref{thm0716_01}.
To treat the case with $\gamma \geq 1$, we first extend the result in Lemma \ref{lem0901} so that the denseness holds for all $\gamma \geq 1$.
For the proof, we use the following observation.

\begin{lemma}
							\label{lem0905_2}
Let $1 < p < \infty$, $-\infty \le S < T \le \infty$, and $n-1 < \theta < n - 1 + p$. If $u \in \widetilde{\fH}_{p, \theta}^\gamma(S, T)$ with $\gamma \ge 1$ and $\cT u = 0$, then $u \in \fH_{p, \theta}^\gamma(S, T)$ and
\begin{equation}
	\label{eq210404}
\left\| u \right\|_{\fH_{p, \theta}^\gamma(S, T)} \le N \left\| u \right\|_{\widetilde{\fH}_{p, \theta}^\gamma(S, T)} 
\end{equation}
where $N = N(n, p, \theta)$ if $\Omega = \bR^n_+$ and $N = N(n, p, \theta, M, K, R)$ if $\partial \Omega \in C^{0, 1}$.
\end{lemma}

\begin{proof}
We first remark that, since $\widetilde{\fH}_{p, \theta}^\gamma(S, T) \subset \widetilde{\fH}_{p, \theta}^1(S, T)$ for $\gamma \geq 1$, the trace operator $\cT$ in the lemma is the one from $\widetilde{\fH}_{p, \theta}^1(S, T)$ to $W_p^{s/2, s}\left( \left(S, T \right) \times \partial\Omega \right)$ in Proposition \ref{prop0712_01} or Corollary \ref{trace_optimal} by viewing $u$ as an element of $\widetilde{\fH}_{p, \theta}^1(S, T)$ (with an extension from $(S, T)$ to $\bR$ as in Remark \ref{finite_time}, if necessary).

To prove the lemma, it suffices to show that $u \in \bL_{p, \theta - p}(S, T)$ and
\begin{equation}
							\label{eq0411_01}
\|u \|_{\bL_{p, \theta - p}(S, T)} \le N\|u\|_{\widetilde{\fH}_{p, \theta}^{\gamma}(S, T)}.
\end{equation}
Indeed, once we have the above inequality, we use Proposition \ref{propertyweighted} (2) and (5) sufficiently many times until we have the estimate \eqref{eq210404} (see the proof of \cite[Lemma 2.14]{MR2386392}).

We now prove \eqref{eq0411_01} only for $\Omega = \bR^n_+$.
The other case can be treated as earlier using a partition of unity and a flattening argument.
Let $u \in \widetilde{\fH}_{p, \theta}^\gamma(S, T)$.
Since $\widetilde{\fH}_{p, \theta}^\gamma(S, T) \subset \widetilde{\fH}_{p, \theta}^1(S, T)$, by Lemma \ref{lem0901}, there exists a sequence $\{u^m\} \subset C_0^\infty(\overline{(S,T) \times \bR^n_+})$ such that $u^m \to u$ in $\widetilde{\fH}_{p,\theta}^1(S,T)$.
By Proposition \ref{prop0712_01} with $\cT u = 0$, we see that $u_m(t,0,x') \to 0$ in $W_p^{s/2,s}((S,T) \times \bR^{n-1})$ as $m \to \infty$.
In particular, $u^m(t,0,x') \to 0$ in  $L_p((S,T) \times \bR^{n-1})$ as $m \to \infty$.
After relabeling, we may assuming that
\begin{equation}
							\label{eq0411_05}
\|u^m(\cdot,0,\cdot)\|_{L_p((S,T) \times \bR^{n-1})} < \frac{1}{m}.\end{equation}
Find an infinitely differentiable function $\zeta(x_1)$ such that $\zeta(x_1) = 0$ for $x_1 \leq 0$ and $\zeta(x_1) = 1$ for $x_1 \geq 1$.
Then set $\zeta_m(x_1) = \zeta(m x_1)$ and
\[
\tilde{u}^m(t,x_1,x') = \zeta_m(x_1)u^m(t,x_1,x').
\]
Since $\tilde{u}^m(t,0,x') = 0$, by Hardy's inequality it follows easily that
\begin{equation}
							\label{eq0411_04}
\|\tilde{u}^m\|_{\bL_{p,\theta-p}(S,T)} \leq N \|D\tilde{u}^m\|_{\bL_{p,\theta}(S,T)},
\end{equation}
where $N = N(n,p,\theta)$.
Thus, to complete the proof of \eqref{eq0411_01}, it only remains to show that
\begin{equation}
							\label{eq0411_03}
\|D_1\tilde{u}^m - D_1 u\|_{\bL_{p,\theta}(S,T)} \to 0
\end{equation}
as $m \to \infty$.
Indeed, because $u^m \to u$ in $\widetilde{\fH}_{p,\theta}^1(S,T)$ it is clear that 
\[
\|\tilde{u}^m - u\|_{\bL_{p,\theta}(S,T)} + \|D_{x'}\tilde{u}^m - D_{x'}u\|_{\bL_{p,\theta}(S,T)} \to 0
\]
as $m \to \infty$.
This combined with \eqref{eq0411_03} and \eqref{eq0411_04} implies
\[
\|u\|_{\bL_{p,\theta-p}(S,T)} \leq N \|Du\|_{\bL_{p,\theta}(S,T)},
\]
which proves  \eqref{eq0411_01} upon recalling that $\|Du\|_{\bL_{p,\theta}(S,T)} \leq \|u\|_{\widetilde{\fH}_{p, \theta}^{\gamma}(S, T)}$ when $\gamma \geq 1$.
To prove \eqref{eq0411_03}, observe that
\[
\|D_1\tilde{u}^m - D_1 u\|_{\bL_{p,\theta}(S,T)} \leq \|\zeta'_m u^m \|_{\bL_{p,\theta}(S,T)} + \|\zeta_m D_1 u^m - D_1 u\|_{\bL_{p,\theta}(S,T)},
\]
where the last term clearly vanishes as $m \to \infty$.
Thus, we prove
\begin{equation}
							\label{eq0411_06}
\|\zeta'_m u^m \|_{\bL_{p,\theta}(S,T)} \to 0 \quad \text{as} \quad m \to \infty.
\end{equation}
Since
\[
|u^m(t,x_1,x')| \leq \int_0^{x_1} |D_1u^m(t,r,x')| \, dr + |u^m(t,0,x')|,
\]
by multiplying both sides of this inequality by $|\zeta'_m(x_1)|$ and using Minkowski's inequality, we get
\[
|\zeta'_m(x_1)|\|u^m(\cdot,x_1,\cdot)\|_p \leq |\zeta'_m(x_1)|\int_0^{x_1} \|D_1u^m(\cdot,r,\cdot)\|_p \, d r
\]
\[
+ |\zeta'_m(x_1)|\|u^m(\cdot,0,\cdot)\|_p,
\]
where $\|\cdot\|_p = \|\cdot\|_{L_p((S,T) \times \bR^{n-1})}$ and by H\"{o}lder's inequality
\[
\int_0^{x_1} \|D_1u^m(\cdot,r,\cdot)\|_p \, dr \leq N x_1^{1 - \frac{\theta-n+1}{p}} \left(\int_0^{x_1} \|D_1u^m(\cdot,r,\cdot)\|_p^p \, r^{\theta - n} \, dr\right)^{1/p}.
\]
Hence,
\[
\|\zeta'_m u^m \|_{\bL_{p,\theta}(S,T)}^p \leq \int_0^\infty |\zeta_m'(x_1)|^p \|u^m(\cdot,0,\cdot)\|_p^p \, x_1^{\theta - n} \, dx_1
\]
\[
+ N \int_0^\infty |\zeta_m'(x_1)|^p x_1^{p-1} \int_0^{x_1} \|D_1u^m(\cdot,r,\cdot)\|_p^p \, r^{\theta - n} \, dr \, dx_1 =: I_1 + I_2,
\]
where by the choice of $\zeta_m$ and \eqref{eq0411_05}
\[
I_1 \leq \int_0^{\frac{1}{m}} m^p m^{-p} x_1^{\theta - n} \, dx_1 \leq N \left(\frac{1}{m}\right)^{\theta - n + 1} \to 0
\]
as $m \to \infty$.
For $I_2$, we have
\[
I_2 = N \int_0^\infty \|D_1u^m(\cdot,r,\cdot)\|_p^p \, r^{\theta - n} \int_r^\infty |\zeta'_m(x_1)|^p x_1^{p-1}  \, dx_1 \, dr
\]
\[
\leq N \int_0^{\frac{1}{m}} \|D_1u^m(\cdot,r,\cdot)\|_p^p \, r^{\theta - n} \, dr \to 0
\]
as $m \to \infty$ because $Du^m \to Du$ in $\bL_{p,\theta}(S,T)$.
This proves \eqref{eq0411_06}.
Therefore, \eqref{eq0411_03} is proved, and so is \eqref{eq0411_01}.
\end{proof}

\begin{lemma}
							\label{lem0603_1}
Let $1< p < \infty$, $-\infty \le S < T \le \infty$ and $n-1 < \theta <n-1+p$.
If $\gamma \ge 1$, $C_0^\infty(\overline{(S,T) \times \Omega})$ is dense in $\widetilde{\fH}_{p,\theta}^\gamma(S,T)$.
\end{lemma}

\begin{proof}
As in the proof of Lemma \ref{lem0905_2}, let $\cT$ be the trace operator from $\widetilde{\fH}_{p, \theta}^1(  S, T)$ to $W_{p}^{\mathsf{s}/2, \mathsf{s}}((S, T) \times \partial\Omega)$.
On the other hand, by Propositions \ref{prop0715_01} and \ref{extension_lip} there is an extension operator $\cS$ such that
$$
 \cS : W_{p}^{\mathsf{s}/2, \mathsf{s}}((S, T) \times \partial\Omega) \to \mathcal{W}_{p, \theta + kp}^{k+1}(  S, T)
$$
for a positive integer $k$ with $k \le \gamma < k+1$.
Indeed, to see this, it is enough to consider $(S, T) = \bR$. 
If $\Omega = \bR^n_+$, recall that $\cS g$ in Proposition \ref{prop0715_01} is not only in $\widetilde{\fH}_{p, \theta}^\gamma(\bR)$ but also in $\mathcal{W}_{p, \theta + kp}^{k+1}(\bR \times \Omega)$ as noted in Remark \ref{rmk0117}. For $\partial \Omega \in C^{0, 1}$, it is also true by the boundedness of $\Omega$.

Using the above operators $\cT$ and $\cS$, for $u \in \widetilde{\fH}_{p,\theta}^\gamma(S,T)$, which is an element of $\widetilde{\fH}_{p,\theta}^1(S,T)$ as well, we have 
$$
\cT u \in W_{p}^{s/2, s}( (S, T) \times \partial\Omega), \quad \cS \cT u \in \mathcal{W}_{p, \theta + kp}^{k+1}( S, T  ) \cap \widetilde{\fH}_{p, \theta}^\gamma(S, T).
$$
Since $\cT \left( \cS \cT u\right) = \cT u$, i.e. $\cT\left(u - \cS \cT u \right) = 0$ and $u - \cS \cT u \in \widetilde{\fH}_{p, \theta}^\gamma(S, T)$, we know $u - \cS \cT u \in \fH_{p, \theta}^\gamma(S, T)$ by Lemma \ref{lem0905_2}. That is, we decompose $u \in \widetilde{\fH}_{p, \theta}^\gamma(S, T)$ by 
$$
u = \left( u - \cS \cT u \right) + \cS \cT u := u_1 + u_2
$$
where $u_1 \in \fH_{p, \theta}^\gamma(S, T)$ and $u_2 \in \mathcal{W}_{p, \theta + kp}^{k+1}(S, T )$.

For approximating $u_1$, it is well known that $C_0^{\infty}\left( [S, T] \times \Omega \right)$ is dense in $\fH_{p, \theta}^\gamma( S, T )$ and $\| u_1 \|_{\widetilde{\fH}_{p, \theta}^{\gamma} (   S, T  )} \le N \| u_1 \|_{\fH_{p, \theta}^{\gamma} ( S, T )}$ when $\Omega$ is bounded.
 If $\Omega = \bR^n_+$, it may not be $\| u_1 \|_{\widetilde{\fH}_{p, \theta}^\gamma( S, T )} \le N \| u_1 \|_{\fH_{p, \theta}^\gamma(S, T)}$, but we make use of Lemma \ref{lem0911} with the fact that $\| v \|_{\widetilde{\fH}_{p, \theta}^\gamma(S, T)} \le N \| v \|_{\fH_{p, \theta}^\gamma(S, T)}$ for $v \in \fH_{p, \theta}^\gamma(S, T)$ such that $v$ vanishes for large $x_1$. On the other hand, for $u_2$, we know $C_0^{\infty}(\overline{ (S, T) \times \Omega}) $ is dense in $\mathcal{W}_{p, \theta + kp}^{k+1}(S, T )$ (Lemma \ref{lem0905_1}) and $\mathcal{W}_{p, \theta + kp}^{k+1}(S, T) \subset \widetilde{\fH}_{p, \theta}^k(S, T) \subset \widetilde{\fH}_{p, \theta}^\gamma(S, T)$ (Lemma \ref{lem0910_01}). 
The lemma is proved.
\end{proof}

Now we are ready to prove Theorem \ref{thm0716_01}.
						
\begin{proof}[Proof of Theorem \ref{thm0716_01}]
First, note that the even extension with respect to the $t$ variable allows us to concentrate on the case $T = \infty$.

To prove the trace part, by Lemma \ref{lem0603_1}, it is enough to prove \eqref{eq_trace} for $u \in C_0^\infty(\overline{(S,T) \times \Omega}) \cap \widetilde{\fH}_{p,\theta}^\gamma(S,T)$, which then follows directly from the inequality \eqref{eq1001} or \eqref{eq1002} with the inclusion $\widetilde{\fH}_{p,\theta}^\gamma(\bR) \subset \widetilde{\fH}_{p,\theta}^1(\bR)$.
The extension part is just Proposition \ref{prop0715_01} and Corollary \ref{extension_lip}.
The theorem is proved.
\end{proof}

\begin{remark}
Since $u \in \widetilde{\fH}_{p,\theta}^1(S,T)$ for $u \in \widetilde{\fH}_{p,\theta}^\gamma(S,T)$ with $\gamma \geq 1$ and $C_0^\infty(\overline{(S,T) \times \Omega})$ is dense in $\widetilde{\fH}_{p,\theta}^1(S,T)$ (see Lemma \ref{lem0901}), we may define $\cT u$ to be the limit in $W_p^{\mathsf{s}/2,\mathsf{s}}((S,T) \times \partial \Omega)$ of the sequence $u^m|_{(S,T) \times \partial \Omega}$, where $u^m \in C_0^\infty(\overline{(S,T) \times \Omega})$ and $u^m \to u$ in $\widetilde{\fH}_{p,\theta}^1(S,T)$.
In this way, we may not need the denseness of $C_0^\infty(\overline{(S,T) \times \Omega})$ in $\widetilde{\fH}_{p,\theta}^\gamma(S,T)$.
Our results show that one can choose a trace defining sequence $\{u^m\} \subset C_0^\infty(\overline{(S,T) \times \Omega})$ with the convergence $u^m \to u$ in $\widetilde{\fH}_{p,\theta}^\gamma(S,T)$.
The result (and its proof) that $C_0^\infty(\overline{(S,T) \times \Omega})$ is dense in $\widetilde{\fH}_{p,\theta}^\gamma(S,T)$ for $\gamma \geq 1$ may be of independent interest.
\end{remark}

\section{Boundary value problems}
	\label{app}

Throughout this section we assume that $\Omega$ is a bounded $C^1$ domain, i.e., $\partial \Omega \in C^1$ as in \cite{MR2111792}.
More precisely, there exist constants $R, K\in(0,\infty)$ so that for any
$x_0 \in \partial \Omega$ there exists a one-to-one continuously differentiable mapping $\Psi$ of
 $B_{R}(x_0)$ onto a domain $G \subset \bR^n$ such that

(i) $G_+:=\Psi(B_{R}(x_0) \cap \Omega) \subset \bR^n_+$ and
$\Psi(x_0)=0$;

(ii)  $\Psi(B_{R}(x_0) \cap \partial \Omega)= G \cap \{y\in
\bR^n:y^1=0 \}$;

(iii) $\|\Psi\|_{C^{1}(B_{R}(x_0))}  \leq K $ and
$|\Psi^{-1}(y_1)-\Psi^{-1}(y_2)| \leq K |y_1 -y_2|$ for any $y_i
\in G$;

(iv) $D\Psi$ is uniformly continuous in $B_{R}(x_{0})$.

Thanks to Theorem \ref{thm0716_01}, we are now able to deal with non-zero boundary value problems for divergence/non-divergence type parabolic equations in weighted Sobolev spaces.
To state and prove the results, we introduce some conditions on the lower-order coefficients $b = (b^1,\ldots,b^n), \tilde{b} = (\tilde{b}^1, \ldots, \tilde{b}^n)$, and $c$ as follows.
\begin{assumption}
	\label{assumption_div}
There exist constants $\Lambda$ and $\varepsilon$ such that
\begin{enumerate}
\item[(i)] for $(t, x) \in (0, T) \times \Omega$,
\[
  \left|b(t, x) \right| + \rho(x)| \tilde{b}(t, x) | +\rho(x)\left| c(t, x)\right|  \le \Lambda;
\]

\item[(ii)] for $x \in \Omega$,
\[
 \limsup_{\rho(x) \to 0}  \sup_{t \in (0, T)}  \rho(x)| \tilde{b}(t, x) | \le \varepsilon.
\]
\end{enumerate}
\end{assumption}

Recall the definition of $W_{p, 0}^{\mathsf{s}/2, \mathsf{s}}((S, T) \times \partial\Omega)$ introduced in Section \ref{setting}.
One can check that $W_p^{\mathsf{s}/2, \mathsf{s}}((S, T) \times \partial\Omega) = W_{p,0}^{\mathsf{s}/2, \mathsf{s}}((S, T) \times \partial\Omega)$ if $\mathsf{s}p < 2$ by a standard approximation argument.

\begin{theorem}[Divergence equation]
							\label{thm0914_01}
Let $1 < p < \infty$, $T \in (0, \infty)$, and $n-1 < \theta < n-1+p $	 with $\mathsf{s} = \left( n - 1 + p - \theta \right)/p$.					
Then there exists a positive constant $\varepsilon$ depending only on $n$, $p$, $\theta$, and $K$ such that under Assumption  \ref{assumption_div}, the equation
\begin{align}
                      \label{eq200914_04}
- u_t + D_i \left(D_iu + b^iu \right) + \tilde{b}^i D_i u + cu &= D_if_i  \\
		    \label{eq200914_05}
 u|_{(0, T) \times \partial \Omega} &= g
\end{align}
admits a unique solution $u\in \widetilde{\fH}^{1}_{p,\theta}(0,T)$ with $u(0,\cdot) = 0$ for any $f_i \in \bL_{p,\theta}(0,T)$, $i = 1, \ldots, n$, and $g \in W_{p, 0}^{\mathsf{s}/2, \mathsf{s}}((0, T) \times \partial\Omega)$. Moreover, for this solution we have
\begin{equation}
                        \label{200914_06}
\|u\|_{\widetilde{\fH}^{1}_{p,\theta}(0,T)} \leq
N\left( \sum_{i=1}^n \| f_i \|_{\bL_{p,\theta}(0,T)} + \|g\|_{W_p^{\mathsf{s}/2, \mathsf{s}}((0, T) \times \partial\Omega)} \right),
\end{equation}
where $N=N(n, p,\theta, T, K, R, \Lambda, \operatorname{diam}(\Omega))$.
\end{theorem}

\begin{remark}
For simplicity's sake, we only deal with the Laplace operator in the equation \eqref{eq200914_04}.
However, Theorem \ref{thm0914_01} can cover an equation with $D_i(a^{ij}D_ju)$ in place of $\Delta u$ as long as the equation is uniquely solvable in the corresponding weighted Sobolev spaces (indeed, $\fH^{1}_{p,\theta}$) with the homogeneous boundary condition, and there are many results on the unique solvability of zero boundary value problems in weighted Sobolev spaces under weaker conditions on $a^{ij}(t,x)$. For instance, one can have $a^{ij}$ in the class of functions with vanishing mean oscillations (VMO).
We refer the reader to \cite{MR2990037, MR3147235, MR3318165, MR3541508, arXiv:1809.01325}.
\end{remark}

\begin{proof}[Proof of Theorem \ref{thm0914_01}]
By Theorem \ref{thm0716_01}, we have an extension $v \in \widetilde{\fH}_{p, \theta}^1(0, T)$ of $g$ such that $v(0, \cdot) = 0$ since $g \in W_{p, 0}^{\mathsf{s}/2, \mathsf{s}}((0, T) \times \partial\Omega)$.
From the definition of $\widetilde{\fH}_{p, \theta}^1(0, T)$ and Assumption \ref{assumption_div} (i), we have 
\[
b^iv \in \bL_{p, \theta}(0, T),\quad \tilde{b}^iD_iv,\,\, cv \in \bL_{p, \theta + p}(0, T) \subset \bH_{p, \theta + p}^{-1}(0, T)
\]
for $i = 1, \ldots, n$. Thus
\[
D_i(f_i - D_iv - b^iv) + v_t - \tilde{b}^iD_iv - cv \in \bH_{p, \theta + p}^{-1}(0,T).
\]
 Now by \cite[Theorem 6.6]{MR3147235}, there exists an $\varepsilon = \varepsilon(n, p, \theta, K)>0$ such that under Assumption \ref{assumption_div} there is a unique solution $w \in \fH_{p, \theta}^1(0, T)$ of the equation
$$
-w_t + D_i \left( D_i w + b^i w \right) + \tilde{b}^iD_iw + cw 
$$
$$
= D_i\left( f_i - D_iv - b^iv \right) + v_t - \tilde{b}^iD_iv - cv := D_i\tilde{f_i}  + \tilde{f_0}
$$
in $(0, T) \times \Omega$ and $w(0, \cdot) = 0$ with the estimate
\begin{equation}
	\label{sol_homo_div}
\|w\|_{\widetilde{\fH}_{p, \theta}^1(0, T)} \le N\|w\|_{\fH_{p, \theta}^1(0, T)}  \le N \left( \| \tilde{f_i} \|_{\bL_{p, \theta}(0, T)} +  \| \tilde{f_0} \|_{\bH_{p, \theta+p}^{-1}(0, T)}   \right),
\end{equation}
where $N = N(n, p, \theta, T, K, \Lambda, \operatorname{diam}(\Omega))$ and the first inequality is due to the boundedness of $\Omega$.
Note that the assumption for $b$ and $c$ in Assumption \ref{assumption_div} (i) is stronger than the corresponding one in \cite{MR3147235}.
Then $u = v + w$ is the desired solution to \eqref{eq200914_04}--\eqref{eq200914_05}. The estimate \eqref{200914_06} follows from \eqref{sol_homo_div} and \eqref{extension}.
\end{proof}

Similarly, we obtain the next theorem with the help of \cite[Theorem 7.7]{MR2990037}. Note that the assumption for $c$ in Assumption \ref{assumption_nondiv} (i) is stronger than the corresponding one in \cite{MR2990037}.

\begin{assumption}
	\label{assumption_nondiv}
There exist constants $\Lambda$ and $\varepsilon$ such that
\begin{enumerate}
\item[(i)] for $(t, x) \in (0, T) \times \Omega$,
\[
 \rho(x)| b(t, x) | +\rho(x)\left| c(t, x)\right|  \le \Lambda;
\]

\item[(ii)] for $x \in \Omega$,
\[
\limsup_{\rho(x) \to 0}\sup_{t \in (0, T)} \rho(x)\left| b(t, x)\right|  \le \varepsilon.
\]
\end{enumerate}

\end{assumption}

\begin{theorem}[Non-divergence equation]
							\label{thm0601_3}
Let $1 < p < \infty$, $T \in (0, \infty)$, and $n-1 < \theta < n-1+p $ with $\mathsf{s} = \left( n - 1 + p - \theta \right)/p$.					
Then there exists a positive constant $\varepsilon$ depending only on $n$, $p$, $\theta$, and $K$ such that under Assumption  \ref{assumption_nondiv}, the equation
\begin{align}
                      \label{eq200914_01}
- u_t + \Delta u + b^iD_iu + cu &= f \\
			\label{eq200914_02}
 u|_{(0, T) \times \partial \Omega} &= g
\end{align}
 admits a
unique solution $u\in \widetilde{\fH}^{2}_{p,\theta}(0,T)$ with $u(0,\cdot) = 0$ for any
  $f\in \bL_{p,\theta + p}(0,T)$ and $g \in W_{p, 0}^{\mathsf{s}/2, \mathsf{s}}((0, T) \times \partial\Omega)$. Morerover, for this solution we have
\begin{equation}
                        \label{200914_03}
\|u\|_{\widetilde{\fH}^{2}_{p,\theta}(0,T)} \leq
N\left(\| f\|_{\bL_{p,\theta + p}(0,T)} + \|g\|_{W_p^{\mathsf{s}/2, \mathsf{s}}((0, T) \times \partial\Omega)} \right),
\end{equation}
where $N=N(n, p,\theta, T, K, R, \Lambda, \operatorname{diam}(\Omega) )$.
\end{theorem}

\begin{remark}
\label{regularity}
In \cite[Theorem 2.10]{MR2111792}, the second author of this paper and Krylov proved that, for given $f \in \bH_{p, \theta + p}^{\gamma}(0, T)$, $\gamma \in \bR$, there is a unique solution $u \in \fH_{p, \theta}^{\gamma+2}(0, T)$ to the equation 
\[
-u_t + a^{ij}D_{ij}u + b^i D_iu + cu = f.
\]
For this, they imposed some regularity conditions (see Assumption 2.3 in \cite{MR2111792}) on the coefficients, but the boundary of the domain is to be only $\partial\Omega \in C^1$ regardless of $\gamma \in \bR$. 
On the other hand, Theorem \ref{thm0716_01} (and Remark \ref{finite_time}) allows an extension $\mathcal{S} g \in \widetilde{\fH}_{p, \theta}^{\gamma}(0, T)$ for any $\gamma \in \bR$, provided that $g \in W_{p}^{\mathsf{s}/2, \mathsf{s}}((0, T) \times \partial\Omega)$ with $\partial\Omega \in C^{0,1}$ and $\mathsf{s} = (p-\theta+n-1)/p$, where $s$ is independent of $\gamma$.
Thus, in Theorem \ref{thm0601_3}, if we further assume that the coefficients $b$ and $c$ are sufficiently smooth (we are assuming $a^{ij} = \Delta$ just for the simplicity) and $\gamma \ge -1$, we have a unique solution $u \in \widetilde{\fH}_{p, \theta}^{\gamma + 2}(0, T)$ to the equation \eqref{eq200914_01}--\eqref{eq200914_02} with $u(0, \cdot) = 0$ for any $g \in W_{p, 0}^{\mathsf{s}/2, \mathsf{s}}((0, T) \times \partial\Omega)$ and $f \in \bH_{p, \theta + p}^{\gamma}(0, T)$. (Here, we need $\gamma \ge -1$ for the well-definedness of $\cT u$.) That is, we obtain higher regularity of a solution to the equation \eqref{eq200914_01}-\eqref{eq200914_02} without any further regularity conditions on $g$ and $\partial\Omega$ other than those used in Theorem \ref{thm0601_3}. 

The same statement holds for Theorem \ref{thm0914_01} since one can change the divergence type operator into non-divergence form if the coefficients are sufficiently smooth.
\end{remark}

\bibliographystyle{plain}

\def\cprime{$'$}


\end{document}